\documentclass[reqno, 12pt]{amsart}
\let\Ssign=\S
\def\S{\Ssign{\hspace{1.25pt}}}
\makeatletter
\let\origsection=\section \def\section{\@ifstar{\origsection*}{\mysection}} 
\def\mysection{\@startsection{section}{1}\z@{.7\linespacing\@plus\linespacing}{.5\linespacing}{\normalfont\scshape\centering\S}}
\makeatother        

\usepackage{amsmath,amssymb,amsthm}
\usepackage{mathrsfs}
\usepackage{mathabx}\changenotsign
\usepackage{dsfont}

\usepackage{xcolor}
\usepackage[backref]{hyperref}
\hypersetup{
    colorlinks,
    linkcolor={red!60!black},
    citecolor={green!60!black},
    urlcolor={blue!60!black}
}

\usepackage[open,openlevel=2,atend]{bookmark}

\usepackage[abbrev,msc-links,backrefs]{amsrefs}

\usepackage{doi}

\renewcommand{\PrintDOI}[1]{\doi{#1}}

\usepackage[T1]{fontenc}
\usepackage{lmodern}
\usepackage[babel]{microtype}
\usepackage[english]{babel}

\usepackage{geometry}
\numberwithin{equation}{section}
\numberwithin{figure}{section}

\usepackage{enumitem}
\def\rmlabel{\upshape({\itshape \roman*\,})}

\def\alabel{\upshape({\itshape \alph*\,})}
\def\Alabel{\upshape({\itshape \Alph*\,})}
\def\nlabel{\upshape({\itshape \arabic*\,})}

\def\tblue{\textrm{blue}}
\def\tred{\textrm{red}}

\usepackage{tikz}
\usepackage{scalerel}

\newsavebox\vdegbox
\savebox\vdegbox{\tikz{
		\draw[black,fill=black] (90:1) circle (.35);
		\draw[black,line width=0.10cm] (210:1) circle (.30);
		\draw[black,line width=0.10cm] (330:1) circle (.30);
		\draw[opacity=0] (0:1.2) circle (0.1);
	}}

\newsavebox\vvbox
\savebox\vvbox{\tikz{
		\draw[black,line width=0.10cm] (90:1) circle (.30);
		\draw[black,fill=black] (210:1) circle (.35);
		\draw[black,fill=black] (330:1) circle (.35);
		\draw[opacity=0] (0:1.2) circle (0.1);
	}}

\newsavebox\pdegbox
\savebox\pdegbox{\tikz{
		\draw[black,line width=0.10cm] (90:1) circle (.30);
		\draw[black,fill=black] (210:1) circle (.35);
		\draw[black,fill=black] (330:1) circle (.35);
		\draw[black,line width=0.28cm ] (210:1) -- (330:1);
		\draw[opacity=0] (0:1.2) circle (0.1);
	}}

\newsavebox\vvvbox
\savebox\vvvbox{\tikz{
		\draw[black,fill=black] (90:1) circle (.35);
		\draw[black,fill=black] (210:1) circle (.35);
		\draw[black,fill=black] (330:1) circle (.35);
		\draw[opacity=0] (0:1.2) circle (0.1);
	}}
\newcommand{\vvv}{\mathord{\scaleobj{1.2}{\scalerel*{\usebox{\vvvbox}}{x}}}}

\newsavebox\evbox
\savebox\evbox{\tikz{
		\draw[black,fill=black] (90:1) circle (.35);
		\draw[black,fill=black] (210:1) circle (.35);
		\draw[black,fill=black] (330:1) circle (.35);
		\draw[black,line width=0.28cm ] (210:1) -- (330:1);
		\draw[opacity=0] (0:1.2) circle (0.1);
	}}

\newsavebox\eebox
\savebox\eebox{\tikz{
		\draw[black,fill=black] (90:1) circle (.35);
		\draw[black,fill=black] (210:1) circle (.35);
		\draw[black,fill=black] (330:1) circle (.35);
		\draw[black,line width=0.28cm ] (90:1) -- (330:1);
		\draw[black,line width=0.28cm ] (90:1) -- (210:1);
		\draw[opacity=0] (0:1.2) circle (0.1);
	}}
\newcommand{\ee}{\mathord{\scaleobj{1.2}{\scalerel*{\usebox{\eebox}}{x}}}}
\newcommand{\piee}{\pi_{\ee}}
\DeclareRobustCommand{\robustee}{\mathord{\scaleobj{1.2}{\scalerel*{\usebox{\eebox}}{x}}}}

\newsavebox\eeebox
\savebox\eeebox{\tikz{
		\draw[black,fill=black] (90:1) circle (.35);
		\draw[black,fill=black] (210:1) circle (.35);
		\draw[black,fill=black] (330:1) circle (.35);
		\draw[black,line width=0.28cm ] (90:1) -- (330:1);
		\draw[black,line width=0.28cm ] (90:1) -- (210:1);
		\draw[black,line width=0.28cm ] (210:1) -- (330:1);
		\draw[opacity=0] (0:1.2) circle (0.1);
	}}

\theoremstyle{plain}
\newtheorem{thm}{Theorem}[section]

\newtheorem{prop}[thm]{Proposition}
\newtheorem{clm}[thm]{Claim}
\newtheorem{prob}[thm]{Problem}
\newtheorem{cor}[thm]{Corollary}
\newtheorem{lemma}[thm]{Lemma}

\theoremstyle{definition}

\newtheorem{dfn}[thm]{Definition}
\newtheorem{definition}[thm]{Definition}
\newtheorem{exmp}[thm]{Example}

\let\eps=\varepsilon
\let\theta=\vartheta
\let\rho=\varrho
\let\phi=\varphi
\let\vsigma=\varsigma
\let\Ups=\varUpsilon

\def\NN{\mathds N}
\def\ZZ{\mathds Z} 
\def\PP{\mathds P}

\def\RR{\mathds R}
 
\def\cA{\mathcal A}

\def\cP{\mathcal P}
\def\cQ{{\mathcal Q}}

\def\cR{\mathcal R}
\def\cS{{\mathcal S}}
\def\cT{{\mathcal T}}
\def\cK{{\mathcal K}}
\def\cU{{\mathcal U}}
\def\cV{{\mathcal V}}
\def\cX{{\mathcal X}}

\def\fA{\mathfrak{A}}
\def\fB{\mathfrak{B}}
\def\fH{\mathfrak{H}}
\def\fR{\mathfrak{R}}

\def\fx{\mathfrak{x}}

\def\ccB{{\mathscr{B}}}
\def\ccC{{\mathscr{C}}}
\def\ccD{{\mathscr{D}}}

\def\ccL{{\mathscr{L}}}
\def\ccM{{\mathscr{M}}}

\def\ccR{\mathscr{R}}

\let\Red=\fR
\let\Blue=\fB

\def\Js{J_\star}
\def\Ks{K_\star}
\def\Ls{L_\star}

\def\cAh{\cA_h}

\def\cPb{\cP_\bullet}
\def\Pb{P_\bullet}

\DeclareMathOperator{\ex}{ex}

\let\polishlcross=\l
\def\l{\ifmmode\ell\else\polishlcross\fi}

\makeatletter
\def\moverlay{\mathpalette\mov@rlay}
\def\mov@rlay#1#2{\leavevmode\vtop{   \baselineskip\z@skip \lineskiplimit-\maxdimen
   \ialign{\hfil$\m@th#1##$\hfil\cr#2\crcr}}}
\newcommand{\charfusion}[3][\mathord]{
    #1{\ifx#1\mathop\vphantom{#2}\fi
        \mathpalette\mov@rlay{#2\cr#3}
      }
    \ifx#1\mathop\expandafter\displaylimits\fi}
\makeatother

\newcommand{\dcup}{\charfusion[\mathbin]{\cup}{\cdot}}
\newcommand{\bigdcup}{\charfusion[\mathop]{\bigcup}{\cdot}}

\def\tand{\ \text{and}\ }
\def\qand{\quad\text{and}\quad}
\def\qqand{\qquad\text{and}\qquad}

\newcommand{\vrhup}[1]{\scaleobj{0.6}{\scalerel*{\rightharpoonup}{#1}}}
\newcommand{\nrhup}{\mathord{\scaleobj{0.6}{\scalerel*{\rightharpoonup}{x}}}}
\newcommand{\wrhup}{\scaleobj{0.6}{\scalerel*{\rightharpoonup}{W}}}
\def\vseq#1{\ThisStyle{  \mathord{\vbox{\offinterlineskip\ialign{    \hfil##\hfil\cr
    $\SavedStyle{}_{\smash{\vrhup#1}}$\cr
    \noalign{\kern-0.7\scriptspace}
    $\SavedStyle#1$\cr}}}}}
\def\seq#1{\ThisStyle{  \mathord{\vbox{\offinterlineskip\ialign{    \hfil##\hfil\cr
    $\SavedStyle{}_{\smash{\nrhup}}$\cr
    \noalign{\kern-0.5\scriptspace}
    $\SavedStyle#1$\cr}}}}}
\def\wseq#1{\ThisStyle{  \mathord{\vbox{\offinterlineskip\ialign{    \hfil##\hfil\cr
    $\SavedStyle{}_{\smash{\wrhup#1}}$\cr
    \noalign{\kern-0.7\scriptspace}
    $\SavedStyle#1$\cr}}}}}

\let\vec=\seq

\def\tinyskip{\vspace{2pt plus 0.7pt minus 0.7pt}}

\let\setminus=\smallsetminus
\let\emptyset=\varnothing

\let\lra=\longrightarrow
\let\to=\lra

\usepackage{xspace}

\makeatletter
\newcommand{\pushright}[1]{\ifmeasuring@#1\else\omit\hfill$\displaystyle#1$\fi\ignorespaces}
\newcommand{\pushleft}[1]{\ifmeasuring@#1\else\omit$\displaystyle#1$\hfill\fi\ignorespaces}
\makeatother

\usepackage{todonotes}

		\usepackage{marginnote}

\begin{document}
\title[Tur\'an density of $5$-cliques in  hypergraphs with quasirandom links]{
Tur\'an density of cliques of order five in $3$-uniform hypergraphs with quasirandom links}

\author[S. Berger]{S\"oren Berger}

\author[S. Piga]{Sim\'on Piga}

\author[Chr. Reiher]{Christian Reiher}
\address{Fachbereich Mathematik, Universit\"at Hamburg, Hamburg, Germany}
\email{\!\!\{soeren.berger\!,simon.piga\!,christian.reiher\!,mathias.schacht\}@uni-hamburg.de}

\author[V. R\"{o}dl]{Vojt\v{e}ch R\"{o}dl}
\address{Department of Mathematics, Emory University, Atlanta, USA}
 \email{vrodl@emory.edu}

\author[M. Schacht]{Mathias Schacht}

\thanks{The first and the fifth author are supported by ERC Consolidator Grant PEPCo 724903.
The second author is supported by ANID/CONICYT Acuerdo Bilateral DAAD/62170017 
through a Ph.D.\ Scholarship. The fourth author is supported by NSF grant DMS 1764385}


\begin{abstract}
	We show that $3$-uniform hypergraphs with the property that all vertices have a quasirandom 
	link graph with density bigger than $1/3$ contain a clique on five vertices. This result is 
	asymptotically best possible.
\end{abstract}

\maketitle

\section{Introduction}
\label{sec:introduction}
We study extremal problems for~$3$-uniform hypergraphs and here, unless stated otherwise, a hypergraph will  always be~$3$-uniform. Recall that given an integer $n$ and a hypergraph~$F$ the extremal number~$\ex(n,F)$ 
is the maximum number of hyperedges that an~$n$-vertex hypergraph can have without containing a copy of~$F$.
It is well known that the sequence $\ex(n,F)/\binom{n}{3}$ converges and the limit 
defines the \emph{Tur\'an density $\pi(F)$}.
Determining~$\pi(F)$ is a central open problem in extremal combinatorics. In fact, even the case 
when $F$ is a clique on four vertices is still unresolved and known as the $5/9$-conjecture of Tur\'an.  

Erd\H os and S\'os~\cite{ErSo82} suggested a variation restricting the problem only to 
those $F$-free hypergraphs that are \emph{uniformly dense} among large sets of vertices.
More precisely, given a hypergraph~$F$, Erd\H os and S\'os asked for the supremum $d\in [0,1]$ such that 
there exist arbitrarily large~$F$-free hypergraphs~$H=(V,E)$ for which every linear sized subset of the vertices induces a hypergraph of density at least~$d$.
Extremal results for uniformly dense hypergraphs in that context were studied in~\cites{BCKMM,GKL,GKV16,MR06,RRS-a,nullpaper}.
For hypergraphs there are several other notions of ``uniform density'' that 
are closely related to the theory of quasirandom hypergraphs  (see, e.g.,~\cites{ACHPS18,Towsner})
and corresponding extremal results were studied in~\cites{Christiansurvey,cherry,RRS-Mantel,RRS-b}. Here,
we shall focus on the following notion.
\begin{dfn}\label{def:ee}\sl
	For a hypergraph $H=(V,E)$ and reals $d\in[0,1]$, $\eta>0$, 
	we say that $H$ is \emph{$(\eta, d, \ee)$-dense} if for all $P$, $Q\subseteq V \times V$ we have 
		\begin{equation}\label{eq:eedef}
			e_{\ee}(P,Q)
			=
			\Big|\Big\{\big((x,y),(y,z)\big)\in\cK_{\ee}(P,Q)\colon \{x,y,z\}\in E\Big\}\Big|
			\geq
			d\,\big|\cK_{\ee}(P,Q)\big|-\eta |V|^3,
		\end{equation}
		where $\cK_{\ee}(P,Q)=\big\{\big((x,y),(y',z)\big)\in P \times Q\colon y=y'\big\}$.
\end{dfn}
For a fixed hypergraph~$F$, we define the corresponding Tur\'an density
\begin{multline}\label{eq:pieedef}
	\pi_{\ee}(F) 
	= 
	\sup\{d\in[0,1] \colon \text{for every $\eta>0$ and $n\in\NN$ there exists an $F$-free,}\\
		\text{$(\eta, d, \ee)$-dense hypergraph with at least $n$ vertices}\}\,.
\end{multline}
In~\cite{cherry} the last three authors obtained a general upper bound for~$\pi_{\ee}(K^{(3)}_{\l})$, 
which turned out to be best possible for all $\l\leq 16$ except for $\l=5$, $9$, and $10$. 

\begin{thm}
\label{thm:K2r}
	For every integer $t\geq 2$ we have
	\begin{align*}
	    \pi_{\ee}(K_{2^t}^{(3)})&\leq\frac{t-2}{t-1}\,.
		\intertext{Moreover, we have \pushQED{\qed}} 
			0
		&=
		\pi_{\ee}(K_4^{(3)})\,,\\
		\tfrac{1}{3}\leq \pi_{\ee}(K_5^{(3)})\leq\tfrac{1}{2}
		&=
		\pi_{\ee}(K_6^{(3)})=\dots=\pi_{\ee}(K_8^{(3)})\,,\\
		\qand\tfrac{1}{2}\leq \pi_{\ee}(K_9^{(3)})\leq \pi_{\ee}(K_{10}^{(3)})\leq\tfrac{2}{3}
		&=
		\pi_{\ee}(K_{11}^{(3)})=\dots=\pi_{\ee}(K_{16}^{(3)})\,.
		\pushQED{\qed}\qedhere\popQED
	\end{align*}
\end{thm}
Here we close the gap for $\piee(K_5^{(3)})$ and show that the known lower bound is best possible.
\begin{thm}[Main result]\label{thm:main}
We have that
	\[
		\piee(K_{5}^{(3)}) = \frac{1}{3}\,.
	\]
\end{thm}

Theorem~\ref{thm:main} has a consequence for hypergraphs with quasirandom links. 
For a hypergraph~$H=(V, E)$ the~\emph{link graph}~$L_H(x)$ of a vertex~$x$
is defined to be the graph with vertex set~$V$ and edge 
set~$\{yz \in V^{(2)}\colon xyz\in E(H)\}$, where for a set $X$ and an integer $k$ we denote by $X^{(k)}$ the set of all $k$-element subsets of $X$. 
Recall that for given~$d\in[0,1]$ and  $\delta >0$ a graph $G=(V,E)$
is said to be~\emph{$(\delta,d)$-quasirandom} if 
for every subset of vertices~$X\subseteq V$ the number of edges~$e(X)$ inside $X$ satisfies
\[
	\bigg|e(X) - 	d\frac{|X|^2}{2}\bigg| \leq  \delta |V|^2\,.
\]
One can check that if all the vertices of a hypergraph~$H$ have 
a~$(\delta, d)$-quasirandom link graph, 
then~$H$ is~$(f(\delta), d, \ee)$-dense, 
where~$f(\delta)\lra0$ as~$\delta\lra0$. In fact, such hypergraphs 
even satisfy in addition a matching upper bound for $e_{\ee}(P,Q)$ 
in~\eqref{eq:eedef} and, hence, having quasirandom links is a stronger property. 
However, the lower bound construction for~$\piee(K_5^{(3)})$ given below
has quasirandom links with density $1/3$ and, therefore, Theorem~\ref{thm:main} 
yields an asymptotically optimal result for such hypergraphs.

\begin{exmp}\label{ex:K5}\sl
For a map $\psi\colon V^{(2)} \longrightarrow \mathds Z /3\mathds Z$ 
we define the hypergraph~$H_\psi=(V,E)$
by
\begin{align}\label{eq:edge}
	xyz\in E
	\qquad\Longleftrightarrow\qquad
	\psi(xy) + \psi(xz)+\psi(zy) \equiv 1 \pmod 3\,.
\end{align}
Observe that for any set of five different vertices~$U=\{u_1, u_2, u_3, u_4, u_5\}$  
double counting yields the identity 
\[
	\sum_{u_iu_ju_k\in {U^{(3)}}} \big(\psi(u_iu_j) + \psi(u_iu_k)+\psi(u_ju_k)\big)
	= 
	3\sum_{u_iu_j\in {U^{(2)}}}\psi(u_iu_j)\,.
\] 
Since the second sum is zero modulo~$3$, at least one of the ten triples in the first sum fails to satisfy~\eqref{eq:edge}.
Consequently,~$H_\psi$ is~$K_5^{(3)}$-free for every map $\psi$.

Moreover, if $\psi$ is chosen uniformly at random, 
then following the lines of the proof of~\cite{cherry}*{Proposition~13.1} shows 
that for every fixed $\delta>0$ and sufficiently large $|V|$
with high probability the hypergraph~$H_\psi$ has the property that all link graphs 
are $(\delta,1/3)$-quasirandom.
\end{exmp}
Summarising the discussion above we arrive at the following corollary, which in light of~Example~\ref{ex:K5} is asymptotically best possible. 
\begin{cor}\label{cor:main}
	For every $\eps>0$ there exist $\delta>0$ and an integer $n_0$ such that every 
	hypergraph on at least $n_0$ vertices all of whose link graphs are 
	$(\delta,1/3+\eps)$-quasirandom 
	contains a copy of $K_5^{(3)}$. \qed
\end{cor}

The proof of Theorem~\ref{thm:main} is based on the regularity method  for hypergraphs. More precisely, 
we shall address the corresponding problem for reduced hypergraphs $\cA$
(see Proposition~\ref{prop:K5reduced}).
The proof of Proposition~\ref{prop:K5reduced} is based on a further reduction 
to the case, when there exists an underlying bicolouring of the pairs~$V^{(2)}$, which corresponds
to a bicolouring of the vertices in the reduced hypergraph $\cA$ (see Proposition~\ref{prop:reduction}).
Finally, we show that in the context of Theorem~\ref{thm:main} such bicoloured reduced hypergraphs yield a~$K_5^{(3)}$ (see Proposition~\ref{prop:bicolored}). 
Sections~\ref{pf-reduction} and~\ref{sec:bicoloured} are devoted to the proofs of Propositions~\ref{prop:reduction} and~\ref{prop:bicolored}.

\section{Reduced hypergraphs and bicolourings}
Similar as in~\cites{RRS-b,RRS-a,nullpaper,cherry} the proof of Theorem~\ref{thm:main} utilises the regularity method for hypergraphs.
This allows us to transfer the problem to an extremal problem for reduced hypergraphs, which play a similar r\^ole for hypergraphs as reduced graphs 
in applications of Szemer\'edi's regularity lemma for graphs.

\begin{definition}\sl
Given a set of indices~$I$ and pairwise disjoint, non-empty sets of vertices~$\mathcal P^{ij}$ for every pair of indices~$ij\in I^{(2)}$, let for every triple of distinct indices~$ijk \in {I^{(3)}}$ a tripartite
 hypergraph~$\mathcal A^{ijk}$ with vertex classes~$\mathcal P^{ij}$,~$\mathcal P^{ik}$, and~$\mathcal P^{jk}$ be given. 
 
 We call the $\binom{\vert I\vert}{2}$-partite hypergraph~$\mathcal A$ defined by 
\[
	V(\mathcal A) = \bigdcup_{ij\in {I^{(2)}}}\cP^{ij} \qqand 
	E(\mathcal A) = \bigdcup_{ijk\in {I^{(3)}}}E(\mathcal A^{ijk})
\]
a \emph{reduced} hypergraph with \emph{index set~$I$}.
Moreover, we say $\cA$ has
\emph{vertex classes~$\mathcal P^{ij}$} and \emph{constituents~$\mathcal A^{ijk}$}. 
\end{definition} 

In this work the index set $I$ will often be an ordered set 
and we may assume $I\subseteq \NN$. 
When we say that a reduced hypergraph is sufficiently large, we mean that its 
index set is sufficiently large. 
Theorem~\ref{thm:main} concerns $\ee$-dense and $K_5$-free hypergraphs~$H$ and next we define the corresponding properties in the context of reduced hypergraphs.

\begin{definition}\label{def:cherrydensity}\sl
For~$d\in[0,1]$ we say that a reduced hypergraph~$\mathcal A$ with index set~$I$ 
is~\emph{$(d,\ee)$-dense}, if for every~$ijk\in I^{(3)}$ and all vertices~$P^{ij}\in\mathcal P^{ij}$ and~$P^{ik}\in\mathcal P^{ik}$ we have
\[
    d\big(P^{ij},P^{ik}\big)
    =
    \big\vert\{P^{jk}\in\mathcal P^{jk}\colon P^{ij}P^{ik}P^{jk} \in E(\mathcal A^{ijk})\} \big\vert 
    \geq 
    d\,\vert\mathcal P^{jk}\vert\,.
\]
\end{definition}

\begin{definition}\sl
We say a reduced hypergraph~$\mathcal A$ with index set $I$ \emph{supports} a clique~$K_\l^{(3)}$ 
if there are an $\l$-element subset $J\subseteq I$ and 
vertices~$P^{ij}\in \mathcal P^{ij}$ for every 
$ij\in J^{(2)}$ such that 
\[
	P^{ij}P^{ik}P^{jk}\in E(\cA^{ijk})
\]
for all $ijk\in J^{(3)}$.
\end{definition}

With these concepts at hand, it follows from~\cite{Christiansurvey}*{Theorem~3.3} that the upper bound in Theorem~\ref{thm:main} is a direct consequence 
of the following statement for reduced hypergraphs.

\begin{prop}\label{prop:K5reduced}
For every~$\eps>0$ every sufficiently large~$\left(\frac{1}{3}+\eps, \ee\right)$-dense reduced hypergraph $\mathcal A$ supports a $K_5^{(3)}$.
\end{prop}

The proof of Proposition~\ref{prop:K5reduced} proceeds by contradiction, so we assume 
that for some~$\eps>0$ there 
are~$(\frac{1}{3}+\eps, \ee)$-dense reduced hypergraphs of unbounded size that do not support $K_5^{(3)}$. 
This motivates the following notion.
\begin{definition}\label{def:wicked}\sl
	For $\eps>0$ we say a reduced hypergraph $\cA$ is \emph{$\eps$-wicked} if 
	it is $(\frac{1}{3}+\eps, \ee)$-dense and fails to support a $K_5^{(3)}$.
\end{definition}

Proposition~\ref{prop:K5reduced} asserts that wicked reduced hypergraphs cannot have 
too many indices
and the proof is divided into two main parts. 
First we reduce the problem to the case in which the reduced hypergraph $\cA$ on some index set $I$ 
can be \emph{bicoloured}.
By this we mean that there is a colouring~$\phi\colon V(\mathcal A)\to\{\tred, \tblue\}$ of the vertices such that for every~$ij\in I^{(2)}$ we have
\begin{equation}\label{eq:nontrivial}
	\phi^{-1}(\tred)\cap \cP^{ij}\neq\emptyset
	\qqand
	\phi^{-1}(\tblue)\cap\cP^{ij}\neq\emptyset
\end{equation}
and there are no hyperedges in $\cA$ with all three vertices of the same colour.
Given such a colouring~$\phi$, we define the \emph{minimum monochromatic codegree density of~$\mathcal A$ and $\phi$} by 
\begin{equation}\label{eq:mmcd-def}
	\tau_2(\mathcal A,\phi)
	=
	\min_{ijk\in I^{(3)}}
	\min	\Big\{
		\frac{d(P^{ij}, P^{ik})}{\mathcal \vert\mathcal P^{jk} \vert}\colon 
		P^{ij}\in \mathcal P^{ij},\ P^{ik}\in \mathcal P^{ik}, \tand \phi(P^{ij})=\phi(P^{ik})
	\Big\}\,.
\end{equation}

The following proposition reduces Proposition~\ref{prop:K5reduced} to bicoloured reduced hypergraphs.

 \begin{prop}\label{prop:reduction}
Given~$\eps>0$ and~$t\in \NN$, let $\cA$ be a sufficiently large $\eps$-wicked reduced hypergraph. 
There exist a reduced hypergraph~$\cA_\star$ with index set of size at least~$t$ 
not supporting a $K_5^{(3)}$
and a bicolouring~$\phi$ of~$\cA_\star$ such that~$\tau_2(\cA_\star, \phi)\geq \frac{1}{3}+\frac{\eps}{8}$. 
\end{prop}

For the proof of Proposition~\ref{prop:reduction} we mainly analyse holes in wicked reduced hypergraphs, i.e., 
subsets of vertices inducing very few edges. It turns out that we can find two ``large''
but almost disjoint holes such that most edges with two vertices in one of the holes 
have their third vertex in the other hole. This configuration can be used to
define an auxiliary reduced hypergraph $\cA_\star$ admitting an appropriate 
colouring~$\phi$ (see Section~\ref{pf-reduction}).

The next proposition completes the proof of Proposition~\ref{prop:K5reduced} by contradicting the conclusion of Proposition~\ref{prop:reduction}, thus showing 
that large wicked hypergraphs indeed do not exist.
\begin{prop}\label{prop:bicolored}
For every~$\eps>0$ every sufficiently large bicoloured reduced hypergraph~$\mathcal A$ 
with~$\tau_2(\mathcal A,\phi) \geq \frac{1}{3}+\eps$ supports a~$K_5^{(3)}$. 
\end{prop}
The proof of Proposition~\ref{prop:bicolored} is deferred to Section~\ref{sec:bicoloured}.

\section{Preliminaries}
\label{sec:preliminaries}

In this section we introduce some necessary definitions and properties for reduced hypergraphs.

\subsection{Transversals and cherries}
We start with the following notion for reduced hypergraphs $\cA$ with index set $I$. 
For~$J\subseteq I$ we refer to a set of 
vertices~$\cQ(J)=\{Q^{ij}\colon ij\in J^{(2)}\}$ with~$Q^{ij}\in \cP^{ij}$ 
for all $ij\in J^{(2)}$ as a~\emph{$J$-transversal}. 
Similarly, for two disjoint subsets of indices~$K,L\subseteq I$ we say 
that~$\cQ(K,L)=\{Q^{k\ell} \colon (k,\ell)\in K\times L\}$ is 
a \emph{$(K,L)$-transversal} when~$Q^{k\ell}\in \cP^{k\ell}$ 
for all $(k,\ell)\in K\times L$. Transversals will always be denoted by  
calligraphic capital letters and the vertices they contain are denoted by the
corresponding Roman capital letters (equipped with a pair of indices as superscript). 

For subsets~$J_\star\subseteq J$, $K_\star\subseteq K$, and~$L_\star\subseteq L$ 
we refer to the transversals $\cQ(J_\star)\subseteq \cQ(J)$ and ~$\cQ(K_\star, L_\star)\subseteq \cQ(K,L)$ (defined in the obvious way) as \emph{restricted transversals}.
Whenever the sets~$J$, $K$, $L\subseteq I$ are clear from the context, we may omit them and write \emph{transversal} to refer to~$J$-transversals or to~$(K,L)$-transversals.

Let us recall that we are often assuming implicitly that our index sets are accompanied 
by a distinguished linear order denoted by $<$. 
Since we are working with~$\ee$-dense reduced hypergraphs (see Definition~\ref{def:cherrydensity}), pairs of vertices sharing one index will play an important r\^ole. 
More precisely, given indices $ijk\in I^{(3)}$ with~$i<j<k$ and given vertices~$P^{ij}\in \cP^{ij}$, $P^{ik}\in \cP^{ik}$, and~$P^{jk}\in \cP^{jk}$ we say that the ordered pair~$(P^{ij},P^{ik})$ is a \emph{left cherry}, the ordered pair~$(P^{ik},P^{jk})$ is a \emph{right cherry}, and the ordered pair~$(P^{ij},P^{jk})$ is a \emph{middle cherry}.
Often we refer to them simply as \emph{cherries}.

For indices~$ijk\in I^{(3)}$ and a set of 
left cherries~$\ccL^{ijk} \subseteq \cP^{ij}\times \cP^{ik}$ we say a 
transversal~$\cQ$ \emph{avoids}~$\ccL^{ijk}$ if~$(Q^{ij}, Q^{ik})\not\in \ccL^{ijk}$ for~$Q^{ij}, Q^{ik}\in \cQ$. Furthermore, we say~$\cQ$ avoids a set of left 
cherries $\ccL=\bigcup_{ijk\in I^{(3)}}\ccL^{ijk}$, if it 
avoids~$\ccL^{ijk}$ for every $ijk\in I^{(3)}$. Similarly, $\cQ$ avoids 
a set of right cherries $\ccR^{ijk}\subseteq \cP^{ik}\times \cP^{jk}$ if 
$(Q^{ik}, Q^{jk})\not\in \ccR^{ijk}$, and $\cQ$ avoids a set of right cherries 
$\ccR=\bigcup_{ijk\in I^{(3)}}\ccR^{ijk}$ if it avoids each $\ccR^{ijk}$. 
Note that these definitions apply both to $J$-transversals and to $(K,L)$-transversals.

\subsection{Inhabited transversals in weakly dense reduced hypergraphs}
We shall utilise a key result from~\cite{nullpaper} on $\vvv$-dense hypergraphs.
Roughly speaking, this notion concerns hypergraphs which have a uniform edge 
distribution on large sets of vertices. However, here we restrict ourselves to the corresponding concepts for 
reduced hypergraphs arising after an application of the hypergraph regularity lemma 
(see, e.g.,~\cites{nullpaper, Christiansurvey} for more details).

\begin{definition}\label{def:threedotdensity}\sl
Let~$\mu>0$ and let~$\mathcal A$ be a reduced hypergraph on an index set~$I$.
We say that~$\mathcal A$ is~\emph{$(\mu,\vvv)$-dense}, if for every~$ijk\in I^{(3)}$ we have
\begin{equation}
    e(\cA^{ijk})
    \ge 
    \mu\,\vert \cP^{ij}\vert \vert \cP^{ik}\vert \vert \cP^{jk}\vert\,.\label{iq:threedotdensity}
\end{equation}
Further, for disjoint subsets of indices~$K,L,M\subseteq I$ we say that~$\mathcal A$ is~\emph{$(\mu,\vvv)$-tridense on~$K,L,M$}, if \eqref{iq:threedotdensity} holds for every triple~$(i,j,k)$ in~$K \times L\times M$.
\end{definition}
Note that by definition every $(d,\ee)$-dense reduced hypergraph is also $(d,\vvv)$-dense.
The following result from~\cite{nullpaper}*{Lemma~3.1} states the existence of transversals containing edges in~$\vvv$-dense reduced hypergraphs. 

\begin{thm}\label{lem:nullpaper}
Let~$t\in \mathds N$, $\mu>0$, and let~$\cA$ be a~$(\mu,\vvv)$-dense reduced hypergraph on 
a sufficiently large index set~$I$. 
There exist a set~$I_\star\subseteq I$ of size~$t$ and three transversals~$\cQ(I_\star)$, 
$\cR(I_\star)$, and~$\cS(I_\star)$ such that~$Q^{ij}R^{ik}S^{jk}\in E(\cA)$ for all~$i<j<k$ in~$I_\star$. \qed
\end{thm}

Triples of transversals satisfying the conclusion of Theorem~\ref{lem:nullpaper} will play an important r\^ole here and this motivates the following definition.

\begin{dfn}[inhabited triple of transversals]
\label{def:inhabited}
\sl
Given a reduced hypergraph~$\cA$ with index set~$I$,
we say a triple of transversals $\cQ(J)\cR(J)\cS(J)$ for some $J\subseteq I$ is \emph{inhabited}
if for all $i<j<k$ in $J$ we have~$Q^{ij}R^{ik}S^{jk}\in E(\cA)$.

Similarly, for pairwise disjoint sets of indices~$K$, $L$, $M\subseteq I$,
we say a triple of transversals $\cQ(K, L)\cR(K, M)\cS(L, M)$  is 
\emph{inhabited} if for every~$k\in K$,~$\ell\in L$, and~$m\in M$ 
we have~$Q^{k\ell}R^{km}S^{\ell m}\in E(\cA)$. 
\end{dfn}

We will also need a version of Theorem~\ref{lem:nullpaper} in which the resulting transversals avoid given sets of forbidden cherries.

\begin{lemma}\label{lem:inhabitedtransversals}
	For all~$t\in \mathds N$ and~$\mu>0$ there is~$\mu'>0$ such that the following holds.
	Let~$\mathcal A$ be a~$(\mu,\vvv)$-dense reduced hypergraph on a sufficiently large 
	index set~$I$ and for all~$i<j<k$ in~$I$ 
	let~$\ccL^{ijk}\subseteq \mathcal P^{ij}\times \mathcal P^{ik}$ 
	and~$\ccR^{ijk}\subseteq\mathcal P^{ik}\times\mathcal P^{jk}$ 
	be sets of left and right cherries satisfying
	\begin{align*}
	    \vert\ccL^{ijk}\vert
  	 	 \leq 
  	  		\mu' \vert\mathcal P^{ij}\vert
  	  		\vert\mathcal P^{ik}\vert 
  	  		\qquad \text{and}\qquad
  	  		\vert\ccR^{ijk}\vert 
  	  	\leq 
    		\mu' \vert\mathcal P^{ik}\vert
    		\vert\mathcal P^{jk}\vert \,.
	\end{align*}
	There exist a set~$I_\star\subseteq I$ of size~$t$ and an inhabited triple 
	of transversals~$\cQ(I_\star)\cR(I_\star)\cS(I_\star)$ avoiding the 
	cherries~$\ccL^{ijk}$ and $\ccR^{ijk}$ for every $ijk \in  I_\star^{(3)}$.

\end{lemma}

For the proof of Lemma~\ref{lem:inhabitedtransversals} we will consider random preimages of reduced hypergraphs. 

\vbox{\begin{definition}\sl
\label{def:randompre}
Given a reduced hypergraph~$\cA$ with index set~$I$ and vertex classes~$\cP^{ij}$
for $ij\in I^{(2)}$, and given an integer $\l\geq 1$,
we fix~$\binom{|I|}{2}$ mutually disjoint sets~$\cPb^{ij}$ of size $\l$
and consider the uniform probability space $\fA(\cA,\l)$ of all 
mappings~$h$ from $\bigcup_{ij\in I^{(2)}}\cPb^{ij}$ to~$\bigcup_{ij\in I^{(2)}}\cP^{ij}$ satisfying 
\[
	h(\cPb^{ij})\subseteq \cP^{ij}
\]
for every $ij\in I^{(2)}$.

With each such map $h$ we associate a reduced hypergraph~$\cAh$ with index set~$I$ and vertex classes~$\cPb^{ij}$ for $ij\in I^{(2)}$
whose edges are defined by 
\[
	\Pb^{ij}\Pb^{ik}\Pb^{jk}\in E(\cAh^{ijk})\quad \Longleftrightarrow\quad
	h(\Pb^{ij})h(\Pb^{ik})h(\Pb^{jk})
	\in 
	E(\cA^{ijk})
\]
for all $ijk\in I^{(3)}$ and all $\Pb^{ij}\in\cPb^{ij}$, 
$\Pb^{ik}\in\cPb^{ik}$, and $\Pb^{jk}\in\cPb^{jk}$.
\end{definition}}

Notice that in this situation $h$ is a hypergraph homomorphism from $\cAh$ to $\cA$.
Below we pass to such a random preimage~$\cAh$ of $\cA$
for sufficiently large $\l$, which will allow us to deduce Lemma~\ref{lem:inhabitedtransversals} for $\cA$ by applying Theorem~\ref{lem:nullpaper}
to $\cAh$.

\begin{proof}[Proof of Lemma~\ref{lem:inhabitedtransversals}]
    Given $t \in \mathds{N}$ and $\mu > 0$,
    let $t_1$ be sufficiently large for an application of Theorem~\ref{lem:nullpaper} with $t$ and $\tfrac{\mu}{2}$ in place of $t$ and $\mu$.
    Further, we fix an integer $\ell$ and $\mu' > 0 $ to satisfy the hierarchy
    \begin{equation*}
          \mu, t_1^{-1} \gg \ell^{-1} \gg \mu'\,,
    \end{equation*}
    where we write $a\gg b$ to indicate that $b$ can be chosen sufficiently small depending on $a$.
    Finally, let $\cA$ be a reduced hypergraph as in the statement of Lemma~\ref{lem:inhabitedtransversals}. We may assume that its index set~$I$ is of size~$t_1$.
            
     Similar as in the proof of~\cite{Christiansurvey}*{Lemma~4.2} we consider the probability space $\fA(\cA,\l)$ from Definition~\ref{def:randompre}
    and we shall prove that with high probability the associated reduced hypergraph~$\cA_h$ is~$(\tfrac{\mu}{2},\vvv)$-dense and no cherry has its image in the sets~$\ccL^{ijk}$ or~$\ccR^{ijk}$.
    
    For every constituent $\cA^{ijk}_h$
    the random variable $e(\cA^{ijk}_h)$ satisfies 
	$\mathds E[e(\cA^{ijk}_h)]\ge\mu\ell^3$ and by Azuma's inequality, applied in the version 
	from~\cite{randomgraphs}*{Corollary~2.27},  we obtain
    \begin{align*}
    \PP\big( \cA_h \text{ is not } (\tfrac{\mu}{2},\vvv)\text{-dense} \big)
    \leq 
    \sum_{ijk \in I^{(3)}}\PP\big(e(\cA^{ijk}_h)<\tfrac{\mu}{2}\ell^3\big)
    \leq 
    \binom{t_1}{3} \exp \bigl(- \tfrac{\mu^2\ell}{24}\bigr)\,.
    \end{align*}
    Moreover, since~$\ccL^{ijk}\leq \mu'\vert \cP^{ij}\vert \vert \cP^{ik}\vert$, the probability that the image of some cherry lies in those sets  is bounded by
    \begin{align*}
        \sum_{ijk \in I^{(3)}} \mathds{P}\Big(h(\Pb^{ij})h(\Pb^{ik})\in \ccL^{ijk} \text{ for some } \Pb^{ij}\Pb^{ik} \in \cPb^{ij} \times \cPb^{ik}\Big) \leq \binom{t_1}{3} \mu' \ell^2\,.
    \end{align*}
    The same inequality holds for the sets~$\ccR^{ijk}$ and our choice of parameters
    ensures 
    \[
    	\binom{t_1}{3}\exp \bigl(-\tfrac{\mu^2\ell}{24}\bigr) + 2 \binom{t_1}{3} \mu'\ell^2< 1\,.
	 \]
    Therefore, we can fix an~$h$ such that~$\cA_{h}$
    is $(\tfrac{\mu}{2},\vvv)$-dense and no cherry has its image in the sets~$\ccL^{ijk}$ or $\ccR^{ijk}$.

    Applying Theorem~\ref{lem:nullpaper} to $\cA_h$ yields a set~$I_\star\subseteq I$ 
    of size~$t$ and a triple of transversals~$\cQ_h(I_\star)\cR_h(I_\star)\cS_h(I_\star)$ 
    inhabited in $\cA_h$.
    It is easy to see that the transversals 
    \[
    	\cQ(I_{\star})=h\bigl(\cQ(I_{\star})\bigr)\,,
		\quad
		\cR(I_{\star})=h\bigl(\cR(I_{\star})\bigr)\,,\qand
		\cS(I_{\star})=h\bigl(\cS(I_{\star})\bigr)
	\]
   are as required.
\end{proof}

\subsection{Partite versions}

We will also need a slightly more involved variant of Theorem~\ref{lem:nullpaper}, 
which guarantees the existence of inhabited triples of transversals in the intersection of multiple $\vvv$-tridense reduced subhypergraphs.

\begin{lemma}\label{lem:crossinhabited}
For all~$t$, $r \in \mathds N$, $\mu>0$ there is some~$s\in \mathds N$ such that the following is true.
Let~$\mathcal A$ be a reduced hypergraph on index set $I$.
Suppose that we have
\begin{enumerate}[label=\alabel]
    \item disjoint subsets of indices $K, L, M \subseteq I$ each of size~$s$,
    \item sets $X_1,\dots, X_r$ of size $s$, and 
    \item for every~$r$-tuple $\vec x \in \prod_{i\in [r]} X_i$ a $(\mu,\vvv)$-tridense subhypergraph~$\mathcal A_{\vec x}\subseteq \mathcal A$ on~$K,L,M$.
\end{enumerate}

Then, there are
\begin{enumerate}[label=\rmlabel]
    \item subsets $K_\star\subseteq K,L_\star\subseteq  L, M_\star\subseteq M$ of size~$t$,
    \item subsets $Y_i\subseteq X_i$ of size~$t$ for every~$i\in [r]$, and 
    \item a triple of transversals $\cQ(K_\star,L_\star)\cR(K_\star, M_\star)\cS(L_\star, M_\star)$, which is inhabited in~$\cA_{\vec y}$ for every~$\vec y\in \prod_{i\in [r]} Y_i$.
\end{enumerate}
\end{lemma}

The proof of Lemma~\ref{lem:crossinhabited} relies on three successive 
applications of the following auxiliary lemma.

\begin{lemma}\label{lem24cn}
For all~$t$, $r \in \mathds N$, $\mu>0$ there is some~$s\in \mathds N$ such that the following is true.
Let~$\mathcal A$ be a reduced hypergraph on index set $I$.
Suppose that we have
\begin{enumerate}[label=\alabel]
    \item disjoint subsets of indices $K, L \subseteq I$ each of size~$s$,
    \item sets $X_1,\dots, X_r$ of size $s$, and 
    \item\label{it:lem24cn-c} for every~$r$-tuple $\vec x \in \prod_{i\in [r]} X_i$, every $k  \in K$, and every $\l \in L$ a subset $\cP_{\vec x}^{kl} \subseteq \cP^{kl}$ of size at least $\mu |\cP^{k \l}|$.
\end{enumerate}

Then, there are
\begin{enumerate}[label=\rmlabel]
    \item subsets $K' \subseteq K,L' \subseteq  L$ of size~$t$,
    \item subsets $X'_i\subseteq X_i$ of size~$t$ for every~$i\in [r]$, and 
    \item a transversal $\cQ( K',L')$ such that for 
    		every~$\vec x\in \prod_{i\in [r]}X'_i$ and every $(k, \l) \in K' \times L'$ 
			we have that~$Q^{k \l}\in \cP_{\vec x}^{k \l}$.
\end{enumerate}
\end{lemma}

\begin{proof}
Given~$t$, $r \in \mathds N, \mu>0$ we fix an integer $s$ such that
\begin{equation}\label{eq:lem24const}
    t,r,\mu^{-1} \ll s\,.
\end{equation}
Let $\cA$ be a reduced hypergraph as in the statement of the lemma and further let $K' \subseteq K$, and $L' \subseteq  L$ be arbitrary subsets of size~$t$.

For every $(K',L')$-transversal $\cQ$ we consider the set
\begin{equation*}
    \fx(\cQ)=\bigg\{ \vec x \in \prod_{i\in [r]} X_i \colon Q^{k \l} \in \cP_{\vec x}^{k \l}\text{ for all } (k, \l) \in K' \times L' \bigg\}\,.
\end{equation*}
Summing over all $(K',L')$-transversals $\cQ$ assumption~\ref{it:lem24cn-c} yields
\begin{equation*}
    \sum_{\cQ} |\fx(\cQ)|
    = 
    \sum_{\vec x \in \prod_{i\in [r]} X_i} \,\, \prod_{(k, \l) \in K'\times L'} 
    	\big|\cP_{\vec x}^{k \l} \big |
    \ge 
    \mu^{t^2} \prod_{(k, \l) \in K'\times L'} \big | \cP^{k \l} \big| 
    	\prod_{i\in [r]} \big | X_i \big|\,.
\end{equation*}
Hence, we can fix a $(K',L')$-transversal $\cQ$ such that
\begin{equation*}
    |\fx(\cQ)| \ge  \mu^{t^2} \prod_{i\in [r]} \big | X_i \big|\,.
\end{equation*}
We may view $\fx(\cQ)$ as an $r$-partite $r$-uniform hypergraph of 
density at least $\mu^{t^2}$ on vertex classes of size $s$. Consequently, 
a result of Erd\H{o}s~\cite{Er64} combined with the hierarchy~\eqref{eq:lem24const}
yields subsets $X'_i\subseteq X_i$ of size~$t$ for every~$i\in [r]$ such that
\begin{equation*}
    \prod_{i\in [r]} X'_i \subseteq \fx(\cQ)\,,
\end{equation*}
which concludes the proof of Lemma~\ref{lem24cn}.
\end{proof}

Next we derive Lemma~\ref{lem:crossinhabited}.

\begin{proof}[Proof of Lemma \ref{lem:crossinhabited}]
Given~$t$, $r \in \mathds N, \mu>0$ we fix integers $s$, $s'$, and $s''$ such that
\begin{equation*}
    t,r,\mu^{-1} \ll s'' \ll s' \ll s
\end{equation*}
and let $\cA$ be a reduced hypergraph as in the statement of the lemma.
We will prove the lemma by applying Lemma \ref{lem24cn} three times, once for every pair from $K$, $L$, and $M$.

\smallskip

{\bf First step.} For every $k \in K$, $\l \in L$, $m \in M$, and every  $\vec x \in \prod_{i\in [r]} X_i$
we consider the set
\begin{equation*}
    \cP_{(\vec x,m)}^{k \l}
    = 
    \Big \{ P^{k \l} \in \cP^{k \l} \colon 
    	\big|N_{\cA_{\vec x}^{k \l m}}(P^{k \l})\big| 
		\ge 
		\tfrac{\mu}{2} |\cP^{k m}||\cP^{ \l m}| 
	\Big \}\,.
\end{equation*}
Since $\cA_{\vec x}$ is $(\mu,\vvv)$-tridense, we have 
\begin{equation*}
    e(A_{\vec x}^{k \l m})
    \geq
    \mu |\cP^{ k \l}| |\cP^{ k m}| |\cP^{ \l m}|\,,
\end{equation*}
and a standard counting argument implies
\begin{equation*}
    \big|\cP_{(\vec x,m)}^{k \l}\big|
    \ge 
    \tfrac{\mu}{2} \big|\cP^{ k \l}\big|\,.
\end{equation*}
Lemma \ref{lem24cn} applied with $s'$, $r+1$, and $\tfrac{\mu}{2}$
in place of $t$, $r$, and $\mu$ and with $X_{r+1}=M$  
yields $s'$-element subsets $K' \subseteq K$, $L' \subseteq  L$, $M' \subseteq  M$, and 
$X'_i\subseteq X_i$ for every~$i\in [r]$ and a transversal $\cQ( K',L')$ such that for 
every~$(\vec x,m) \in\prod_{i\in [r]}X'_i\times M'$ and every $(k, \l)\in K' \times L'$ 
we have that~$Q^{k \l}\in \cP_{(\vec x,m)}^{k \l}$.

\smallskip

{\bf Second Step.} Next we consider for every $k \in K'$, $\l \in L'$, $m \in M'$, and every $\vec x \in \prod_{i\in [r]} X'_i$ the set
\begin{equation*}
    \cP_{(\vec x,\l)}^{k m}
    = 
    \Big \{ P^{k m} \in \cP^{k m} \colon 
    	\big|N_{\cA_{\vec x}^{k \l m}}(Q^{k \l}, P^{k m})\big| 
    	\ge 
    	\tfrac{\mu}{4} |\cP^{\l m}|
	\Big \}\,.
\end{equation*}
By our choice of the transversal $\cQ( K',L')$ we have 
\[
	|N_{\cA_{\vec x}^{k \l m}}(Q^{k \l})|
	\geq
	\tfrac{\mu}{2} |\cP^{k m}||\cP^{ \l m}| 
\]
and, as before, this implies
\begin{equation*}
	|\cP_{(\vec x,\l)}^{k m}| \ge \tfrac{\mu}{4} |\cP^{ k m}|\,.
\end{equation*}
Again, we apply Lemma \ref{lem24cn}, now with $s''$, $r+1$, and $\tfrac{\mu}{4}$
 in place of $t$, $r$, and $\mu$ and with $X'_{r+1}=L'$,
 to reach $s''$-element subsets $K'' \subseteq K'$, $L'' \subseteq  L'$, $M'' \subseteq  M'$, and $X''_i\subseteq X'_i$,  for every~$i\in [r]$
 and a transversal $\cR( K'',M'')$ such that for every~$(\vec x, \l) \in \prod_{i\in [r]}X''_i \times L''$ and every $(k, m) \in K'' \times M''$ we have~$R^{k m}\in \cP_{(\vec x,\l)}^{k m}$.

\smallskip

{\bf Third step.} Last, we consider for every $\l \in L'', k \in K'', m \in M''$, and every $\vec x \in \prod_{i\in [r]} X''_i$ the set
\begin{equation*}
    \cP_{(\vec x, k)}^{\l m}
    = N_{\cA_{\vec x}^{k \l m}}(Q^{k \l}, R^{k m})\,.
\end{equation*}
By our choice of the transversals $\cQ(K'',L'')$ and $\cR(K'',M'')$ we have  
$|\cP_{(\vec x, k)}^{\l m}| \ge \tfrac{\mu}{4} |\cP^{\l m}|$.
The final application of Lemma \ref{lem24cn}, with $t$, $r+1$, and $\tfrac{\mu}{4}$
 in place of $t$, $r$, and $\mu$, yields $t$-sized subsets $K_\star \subseteq K''$, $L_\star \subseteq  L''$, $M_\star\subseteq  M''$, and $Y_i\subseteq X''_i$,  for every~$i\in [r]$, and a transversal $\cS( L_\star,M_\star)$ such that for every~$\vec y \in \prod_{i\in [r]}Y_i$ and every $(k, \l, m) \in K_\star \times L_\star \times M_\star$ we have that $Q^{k \l}R^{k m}S^{\l m}\in E(\cA_{\vec y})$. In other words, the triple 
of transversals $\cQ\cR\cS$ is inhabited in every $\cA_{\seq y}$ with $\seq y\in \prod_{i\in [r]}Y_i$.
\end{proof}

\section{Bicolouring wicked reduced hypergraphs}
\label{pf-reduction}

\subsection{Plan}
This entire section is devoted to the proof of Proposition~\ref{prop:reduction}. 
As the argument is quite long, we would like to commence with a brief 
outline of our strategy. 

\subsubsection{Na\"{i}ve ideas} 
In an attempt to keep this account sufficiently digestible 
we will systematically oversimplify and most claims below will later turn 
out to be true in a metaphorical sense only. 

It might be helpful to know, what the proof of Proposition~\ref{prop:reduction}
does, when the given $\eps$-wicked reduced hypergraph $\cA$ itself possesses   
a bicolouring $\phi$ of the vertices (which might be ``unknown`` to us) 
and to contrast this situation with the general case. 

What can immediately be seen is that in the bicoloured case $\cA$ contains 
many holes, by which we mean that there are many independent 
sets $\Phi\subseteq V(\cA)$ such that for every pair of indices~$ij$ 
we have $|\cP^{ij}\cap\Phi|\ge (1/3+\eps)|\cP^{ij}|$. 
Indeed, there are ``red holes'' consisting of red vertices only and, 
similarly, there are ``blue holes''. For the sake of discussion we will 
pretend that these are the only holes, i.e., that each hole is either red or blue. 

\vbox{In the general case, it might not be clear at first sight that any holes exist, 
but based on the assumption that $\cA$ fails to support a $K_5^{(3)}$ one can 
establish that they do. As a matter of fact, there is  
a fairly flexible method to construct holes and thus one should 
think of the set $\fH$ of all holes in~$\cA$ as having a possibly intricate 
structure. There are three main lemmata in our analysis of $\fH$: 
\begin{enumerate}
	\item[$\bullet$] the transitivity lemma;
	\item[$\bullet$] the union lemma;
	\item[$\bullet$] and the density increment lemma. 
\end{enumerate}}

Let us briefly summarise the content of these three statements.  

{\bf I.} Returning to the bicoloured case, ``being of the same colour'' is an obvious 
equivalence relation on $\fH$, which has two equivalence classes. Moreover, 
if $\phi$ and thus the colouring of the holes is unknown, this equivalence 
relation is definable by saying that two holes are equivalent if and only if they 
intersect each other (substantially) on every vertex class~$\cP^{ij}$. 
When $\cA$ is arbitrary, the relation of intersecting each other in this sense 
is clearly reflexive and symmetric. The aforementioned {\it transitivity lemma}  
ensures that this relation is transitive as well; 
its proof requires some effort.
One can also show that~$\fH$ always consists of exactly two equivalence classes. 

{\bf II.} In the bicoloured case, the union of two red holes is again a red hole and, in 
fact, the class of red vertices is definable as the union of all red holes. 
It turns out that in the general case one can prove the union of two equivalent 
holes to be a hole as well, and this is what the {\it union lemma} asserts. 

{\bf III.} Iterative applications of the union lemma yield two maximal holes, 
namely the unions of the two equivalence classes. In the bicoloured case every 
vertex belongs to one of these maximal holes, but this is not necessary for the 
proof of Proposition~\ref{prop:reduction} to go through. All that matters is that 
for two appropriate holes
\[ 
\text{ most edges with two vertices in one hole have their third vertex
in the other hole.} \tag{$*$}
\]
The density increment lemma states that if two holes violate $(*)$, 
then there are two other holes covering more space. Thus iterative 
applications of this lemma show 
that the maximal holes satisfy $(*)$. 
 
Notice that if we managed to arrive at two holes satisfying~$(*)$, 
then the proof of Proposition~\ref{prop:reduction} could be completed by deleting the 
vertices not belonging to them. 

\subsubsection{A more realistic picture} Let us now point to two deficiencies 
of the foregoing outline. First, we will never show that the given reduced 
hypergraph $\cA$ contains a hole containing no edges at all. All we need and prove is that there 
are large sets inducing very few edges in $\cA$ and thus there will be 
parameters $\mu$, $\nu$, etc.\ quantifying how accurate our holes are. 
Second, each step of the argument 
is accompanied by a significant loss of the relevant part of the index set. 
Thus the number of times we apply our key lemmata needs to be bounded by 
a function of $\eps$ and, therefore, we will never reach holes 
that are maximal in the absolute sense. All that can realistically be said is that 
there are two holes which cannot be enlarged by a substantial amount, 
and for this reason we adopt the somewhat indirect density increment formulation 
of the third main lemma.
    
\subsubsection{Organisation} In~\S\ref{sec:holes}
and~\S\ref{sec:inh-trans} we deal with general properties of $\ee$-dense 
reduced hypergraphs not supporting a $K_5^{(3)}$, including the existence 
of holes. 
The main result of~\S\ref{Equivalent holes} is the transitivity lemma, 
a precise version of which will be stated as Lemma~\ref{lem:transitivity}. 
Next, the union lemma is obtained in~\S\ref{sec:union} (see Lemma~\ref{lem:union}). 
The proof of the density increment lemma (Lemma~\ref{lem:bad}) requires some 
preparations provided in~\S\ref{sec:hdtt}, while the proof itself is given 
in~\S\ref{sec:bicolourisation}. 
Finally, we argue in~\S\ref{ssec:pf-reduction} that despite the approximate nature 
of the arguments provided so far the proof of Proposition~\ref{prop:reduction} 
can be completed by taking a random preimage.
 
\subsection{Holes and links in reduced hypergraphs}
\label{sec:holes}
Given a reduced hypergraph $\cA$ with index set~$I$, a natural definition 
of a \emph{hole} across a subset of indices $J\subseteq I$ and subsets of 
vertices $\Phi^{ij}\subseteq \cP^{ij}$
for $ij\in J^{(2)}$ would maybe require that 
for every $ijk\in J^{(3)}$ the sets 
$\Phi^{ij}$, $\Phi^{ik}$, $\Phi^{jk}$ span no hyperedges
in $\cA^{ijk}$. However, this notion is too restrictive for our analysis 
and we shall only require that these sets induce hypergraphs of low density.

\begin{definition}\label{def:hole}\sl
Given a reduced hypergraph~$\cA$ and a subset of indices~$J\subseteq I$ we say that a subset of vertices~$\Phi\subseteq V(\cA)$ is a~\emph{$\mu$-hole on $J$} if $\Phi^{ij}=\Phi\cap \cP^{ij}$ is nonempty for all $ij\in J^{(2)}$ 
and
\[
	e(\Phi^{ij}, \Phi^{ik}, \Phi^{jk}) \leq \mu \vert \cP^{ij} \vert\vert \cP^{ik} \vert\vert \cP^{jk} \vert
\]
for every~$ijk\in J^{(3)}$.

The \emph{size} of the hole is $|J|$ and the smallest $\vsigma>0$ such that 
$|\Phi^{ij}|\geq \vsigma |\cP^{ij}|$ for every $ij\in J^{(2)}$ is called the \emph{width} of the hole. 
We refer to $\mu$-holes with width at least $\vsigma$ as $(\mu, \vsigma)$-holes.
\end{definition}

Roughly speaking, for the proof of Proposition~\ref{prop:reduction} we shall find two almost disjoint holes with widths bigger than $1/3$ on a large set of indices in a wicked reduced hypergraph.

Holes may induce a few hyperedges. However, cherries that are contained in too many 
such hyperedges are considered to be exceptional. This leads to the following definition.

\begin{definition}\label{def:exceptional-cherries}\sl
	Given a $\mu$-hole $\Phi$ on $J$, $\eps>0$, and~$ijk$ in~$J^{(3)}$
	a cherry~$(P^{ij}, P^{ik})\in \Phi^{ij}\times \Phi^{ik}$ 
	is \emph{$\eps$-exceptional} if 
	\[
		\big|N(P^{ij}, P^{ik})\cap \Phi^{jk}\big|
		\geq 
		\eps \big|\cP^{jk}\big|\,.
	\]
	For indices~$i<j<k$ in~$J$ we denote by 
	\[
		\ccL^{ijk}(\Phi,\eps)\subseteq \cP^{ij}\times\cP^{ik}\,,
		\quad
		\ccM^{ijk}(\Phi,\eps)\subseteq \cP^{ij}\times\cP^{jk}\,,
		\qand
		\ccR^{ijk}(\Phi,\eps)\subseteq \cP^{ik}\times\cP^{jk}\,,
	\]
	the $\eps$-exceptional left, middle, and right cherries and we set
	\[
		\ccL(\Phi,\eps)=\!\bigdcup_{i<j<k}\!\ccL^{ijk}(\Phi,\eps)\,,
		\ \
		\ccM(\Phi,\eps)=\!\bigdcup_{i<j<k}\!\ccM^{ijk}(\Phi,\eps)\,,
		\ \tand\
		\ccR(\Phi,\eps)=\!\bigdcup_{i<j<k}\!\ccR^{ijk}(\Phi,\eps)\,.
	\]
	\end{definition}
It is easy to see that holes can only contain few exceptional cherries. More precisely, 
for every $\mu$-hole $\Phi$ on $J$ and every $\eps>0$ we have for all $i<j<k$ in $J$
\[
    	\eps\,\vert \cP^{jk}\vert \vert\ccL^{ijk}(\Phi,\eps)\vert 
        \leq 
        e(\Phi^{ij}, \Phi^{ik}, \Phi^{jk})
        \leq 
        \mu\,\vert \cP^{ij}\vert \vert \cP^{ik}\vert \vert \cP^{jk}\vert
\]
and the same reasoning applies to $\ccR$ and $\ccM$. This shows
\begin{multline}
    \label{eq:except-cherries}
	\big|\ccL^{ijk}(\Phi,\eps)\big|
	\leq 
	\frac{\mu}{\eps}\big| \cP^{ij}\big|\big| \cP^{ik}\big|\,,
	\quad
	\big|\ccM^{ijk}(\Phi,\eps)\big|
	\leq 
	\frac{\mu}{\eps}\big| \cP^{ij}\big|\big| \cP^{jk}\big|\,,\\
	\qand
	\big|\ccR^{ijk}(\Phi,\eps)\big|
	\leq 
	\frac{\mu}{\eps}\big| \cP^{ik}\big|\big| \cP^{jk}\big|\,.
\end{multline}

Often we consider holes $\Phi$ arising from neighbourhoods 
$N(P^{ik},P^{jk})$, i.e., for appropriately chosen
$P^{ik}\in\cP^{ik}$ and $P^{jk}\in\cP^{jk}$ we set $\Phi^{ij}=N(P^{ik},P^{jk})$. Note that in $(d,\ee)$-dense reduced hypergraphs, 
holes obtained in this way will automatically have width at least~$d$.

Given a $(K, L)$-transversal~$\cQ$, a subset $K_{\star}\subseteq K$, and an 
index~$\l\in  L$ we define the \emph{$\cQ$-link of~$\l$ on~$K_\star$} 
by
\[
	\Lambda(\cQ, K_\star, \l) = \bigdcup_{kk'\in K_\star^{(2)}} N(Q^{k\l},Q^{k'\l})\,.
\]
The following lemma asserts that in $\ee$-dense reduced hypergraphs that do not 
support~$K_5^{(3)}$ the $\cQ$-links contain large holes.

\begin{lemma}\label{lem:lamda}
Let~$t\in\NN$, $\mu$, $d>0$, let~$\cA$ be a~$(d, \ee)$-dense reduced hypergraph with 
index set~$I$ that does not support a~$K_5^{(3)}$, and for sufficiently large disjoint 
subsets of indices~$K,L\subseteq I$ let~$\cQ$ be a $(K, L)$-transversal. 

Then there exist~$K_\star\subseteq K$ and~$L_\star\subseteq L$ of size~$t$ 
such that for every~$\ell\in L_\star$ the link~$\Lambda(\cQ, K_\star, \ell)$ 
is a~$(\mu,d)$-hole .
\end{lemma}

\begin{proof}
    Let~$q=\binom{\lceil \mu^{-1}\rceil}{2}$ and 
    set $\Ups^{kk'}_\ell=N(Q^{k\ell},Q^{k'\ell})\subseteq \cP^{kk'}$ for all $kk'\in K^{(2)}$, $\ell\in L$,
    and, similarly, $\Ups^{\ell\ell'}_k=N(Q^{k\ell},Q^{k\ell'})\subseteq \cP^{\l\l'}$ 
    for all $k\in K$, $\ell\ell'\in L^{(2)}$. 
    Consider an auxiliary $2$-colouring
    of the pairs~$(kk'k'',\ell)\in K^{(3)}\times L$ depending on whether
    \begin{align}\label{eq:nothole}
        e\big(\Ups^{kk'}_\ell, \Ups^{kk''}_\ell, \Ups^{k'k''}_\ell\big) 
        > 
        \mu \vert \cP^{kk'}\vert \vert \cP^{kk''}\vert \vert \cP^{k'k''}\vert
    \end{align}
    holds or not.
    Since~$K$ and~$L$ are sufficiently large, the product Ramsey theorem 
    (see, e.g.,~\cite{Ramseybook}*{Theorem~5.1.5}) 
    yields a set~$K_1\subseteq K$ with 
    $|K_1|\ge\max\{3d^{-q},t\}$ and a set~$L_1\subseteq L$ 
    with $|L_1|\geq \max\{\lceil \mu^{-1}\rceil ,t\}$
    such that all pairs $(kk'k'',\ell)\in K_1^{(3)}\times L_1$
    agree whether~\eqref{eq:nothole} holds or not. 
    In fact, if~\eqref{eq:nothole} fails on $K_1^{(3)}\times L_1$, 
    then arbitrary $t$-element subsets $K_\star\subseteq K_1$ 
    and $L_\star\subseteq L_1$ have the desired property.
    Consequently, we may assume that~\eqref{eq:nothole} holds on~$K_1^{(3)}\times L_1$.
    We shall show that this implies $\cA$ supports a $K_5^{(3)}$. 
    
    Let $L_2$ be a subset of $L_1$ of size~$\vert L_2\vert=\lceil \mu^{-1}\rceil$ 
    and consider some~$\ell\ell'\in L_2^{(2)}$.  
    Since we have~$\vert \Ups^{\ell\ell'}_k\vert \geq d\vert \cP^{\ell\ell'} \vert$ 
    for every~$k\in K_1$, there is a subset~$K_2\subseteq K_1$ of size 
    at least~$d\vert K_1\vert$ such that 
    \[
    	\bigcap_{k\in K_2}\Ups^{\ell\ell'}_k\neq\emptyset\,.
    \]
    Repeating this argument iteratively~$q=\binom{|L_2|}2$ times, once for every pair 
    in~$L_2$, we obtain 
    nested subsets~$K_1\supseteq K_2 \supseteq \dots \supseteq K_{q+1}$ such that    
    \begin{align*}
    	\vert K_{q+1}\vert \geq d^q\vert K_1 \vert\geq 3
		\qand
        \bigcap_{k\in K_{q+1}} \Ups^{\ell\ell'}_k\neq \emptyset \ \text{for every}\ \ell\ell'\in L_{2}^{(2)}\,.
    \end{align*}
        The first statement allows us to fix some~$kk'k''\in K_{q+1}^{(3)}$
        and the second one yields for every~$\ell\ell'\in L_2^{(2)}$ 
        a fixed vertex~$P^{\ell\ell'}\in \cP^{\ell\ell'}$ satisfying
        \begin{align}\label{eq:lamdaP1}
            P^{\ell\ell'}Q^{k \ell}Q^{k\ell'},\, 
            P^{\ell\ell'}Q^{k'\ell}Q^{k'\ell'},\, 
            P^{\ell\ell'}Q^{k''\ell}Q^{k''\ell'} 
            \in E(\cA)\,.
        \end{align}

We infer from~\eqref{eq:nothole} and our choice of the size of $L_2$ that
\begin{align*}
	\sum_{\ell\in L_2} e\big(\Ups^{kk'}_\ell, \Ups^{kk''}_\ell, \Ups^{k'k''}_\ell\big)
	 >
	\mu|L_2|\vert \cP^{kk'}\vert \vert \cP^{kk''}\vert \vert \cP^{k'k''}\vert 
	\ge \vert \cP^{kk'}\vert \vert \cP^{kk''}\vert \vert \cP^{k'k''}\vert\,.
\end{align*}
Consequently, there is an edge $R^{kk'}R^{kk''}R^{k'k''}\in E(\cA^{kk'k''})$
and there exist two distinct indices $\l, \ell'\in L_2$ such that both $\lambda\in\{\ell, \ell'\}$ 
satisfy 
\begin{align*}\label{eq:R1}
	R^{kk'}Q^{k\lambda}Q^{k'\lambda},\,
	R^{kk''}Q^{k\lambda}Q^{k''\lambda},\,
	R^{k'k''}Q^{k'\lambda}Q^{k''\lambda}
	&\in E(\cA)\,.
\end{align*}
Together with~\eqref{eq:lamdaP1} we arrive at the contradiction 
that $P^{\l\l'}$, the six 
vertices~$Q^{\kappa\lambda}$ with $\kappa\in \{k,k',k''\}$ and $\lambda\in\{\l,\l'\}$, 
and the three vertices $R^{kk'}$, $R^{kk''}$, $R^{k'k''}$ support a~$K_5^{(3)}$ in~$\cA$.
\end{proof}

Two consecutive applications of Lemma~\ref{lem:lamda} yield a symmetric conclusion.

\begin{cor}\label{cor:lamda}
Let~$t\in\NN$, $\mu$, $d>0$, let~$\cA$ be a~$(d, \ee)$-dense reduced hypergraph with index set~$I$ that does not support a~$K_5^{(3)}$, and for sufficiently large disjoint subsets of indices~$K,L\subseteq I$ let~$\cQ$ be a $(K, L)$-transversal. 

Then there exist~$K_\star\subseteq K$ and~$L_\star\subseteq L$ of size~$t$ 
such that for every~$\ell\in L_\star$ and for every~$k\in K_\star$ the $\cQ$-links 
$\Lambda(\cQ, K_\star, \ell)$ and $\Lambda(\cQ, L_\star, k)$ are $(\mu,d)$-holes.
\end{cor}
\begin{proof}
	For sufficiently large $t'=t'(t,\mu,d)$ a first application of Lemma~\ref{lem:lamda} yields 
	subsets~$K'$ and $L'$ of size at least $t'$ such that~$\Lambda(\cQ, K', \ell)$ is 
	a~$(\mu,d)$-hole for every~$\ell\in L'$. A second application 
	to the restricted transversal $\cQ(K',L')$ (with the r\^oles of~$K$ and~$L$ exchanged)
	then yields subsets $L_{\star}\subseteq L'$ and $K_{\star}\subseteq K'$ of size $t$ such that 
	additionally~$\Lambda(\cQ, L_\star, k)$ is a~$(\mu,d)$-hole for every~$k\in K_\star$.
\end{proof}

\subsection{Intersecting and disjoint links}
\label{sec:inh-trans}

Next we define concepts for pairs of links  having a substantial intersection and of 
being almost disjoint.

\begin{dfn}\sl
	Let~$\cA$ be a reduced hypergraph with index set $I$, let~$K,L,M\subseteq I$ be 
	pairwise disjoint sets of indices,
	and let~$\cQ(K, L)$ and $\cR(K,M)$ be transversals. 
	
	For $\l\in L$ and $m\in M$ we say 
	the links $\Lambda(\cQ,K,\l)$ and $\Lambda(\cR,K,m)$ are \emph{$\delta$-intersecting}
	if 
	\begin{equation}\label{eq:intersecting}
		\big|N(Q^{k\l},Q^{k'\l}) \cap N(R^{km},R^{k'm})\big|
		\ge
		\delta \big|\cP^{kk'}\big|
	\end{equation}
	for all $kk'\in K^{(2)}$.
	If, on the other hand,~\eqref{eq:intersecting} fails for all $kk'\in K^{(2)}$, then we say 
	$\Lambda(\cQ,K,\l)$ and $\Lambda(\cR,K,m)$ are \emph{$\delta$-disjoint}.
	
	Moreover, we say a pair of transversals $\cQ(K, L)\cR(K,M)$ has \emph{$\delta$-intersecting links}
	(resp.\ \emph{$\delta$-disjoint links})
	if $\Lambda(\cQ,K,\l)$ and $\Lambda(\cR,K,m)$ are $\delta$-intersecting (resp.\ $\delta$-disjoint) 
	for every $\l\in L$ and $m\in M$.
\end{dfn}

We remark that the notions of being $\delta$-intersecting and $\delta$-disjoint 
do not complement each other.
However, by means of (the product version of) Ramsey's theorem we can always pass 
to subsets of $K$, $L$, and $M$ for which one of the properties holds (see, e.g., 
the proof of Corollary~\ref{cor:greekleters} below).

The next lemma and its corollary show that in reduced hypergraphs that do not support~$K_5^{(3)}$ 
at most one pair from a triple of inhabited transversals can have an intersecting 
link (see Definition~\ref{def:inhabited}).
 
\begin{lemma}\label{lem:greekleters}
Let~$\delta >0$, let~$\cA$ be a reduced hypergraph with index set~$I$, and for sufficiently large disjoint sets~$K,L,M\subseteq I$ let~$\cQ(K,L)\cR(K,M)\cS(L,M)$ be an inhabited  triple of transversals. 
If both pairs of transversals~$\cQ(K,L)\cR(K,M)$ and~$\cQ(K,L)\cS(L,M)$ have $\delta$-intersecting links,
then~$\cA$ supports a~$K^{(3)}_5$. 
\end{lemma}

\begin{proof}
Fix~$m\in M$, a subset~$K_\star\subseteq K$ of size $\lfloor \delta^{-1}\rfloor+1$, and~$q=\binom{\lfloor \delta^{-1}\rfloor +1}{2}$.
Consider an arbitrary pair of distinct indices~$k,k'\in K_\star$.
Since $\vert N(Q^{k\ell},Q^{k'\ell})\cap N(R^{km},R^{k'm})\vert \geq \delta \vert\cP^{kk'}\vert$ 
for every~$\ell\in L$, there is a subset~$L_1\subseteq L$ of size at least~$\delta \vert L \vert$ such that
\begin{align}\label{eq:sharer}
    \bigcap_{\ell \in L_1} N(Q^{k\ell},Q^{k'\ell})\cap N(R^{km},R^{k'm})\neq \emptyset\,.
\end{align} 
As~the pair~$kk'$ was taken arbitrarily, we can repeat the argument iteratively~$q$ times (once for every pair in~$K_\star^{(2)}$) and find nested subsets~$L\supseteq L_1\supseteq L_2 \supseteq \dots \supseteq L_q$ such that~\eqref{eq:sharer} with~$L_1$ replaced by $L_q$ holds for every~$kk'\in K^{(2)}_\star$.

Moreover, we have $|L_q|\geq \delta^q|L|$ and,
since~$L$ is sufficiently large, this yields~$|L_q|\geq 2$ 
and we can select $\ell\ell'\in L^{(2)}_q$. 
Owing to~\eqref{eq:sharer} with $L_1$ replaced by $L_q$, 
for every~$kk'\in K_\star^{(2)}$ there is a vertex~$P^{kk'}\in \cP^{kk'}$ such that
\begin{align}\label{eq:P1}
 P^{kk'}Q^{k\ell}Q^{k'\ell},\, 
 P^{kk'}Q^{k\ell'}Q^{k'\ell'},\, 
 P^{kk'}R^{km}R^{k'm}\in E(\cA)\,.
\end{align}

Next, since~$\cQ(K, L)\cS(L, M)$ has $\delta$-intersecting links 
and~$|K_\star|>\delta^{-1}$, there exists a pair~$kk'\in K^{(2)}_\star$ such that
\[
	N(Q^{k\ell},Q^{k\ell'})\cap N(Q^{k'\ell},Q^{k'\ell'}) \cap N(S^{\ell m},S^{\ell'm})
	\neq \emptyset\,.
\]
Therefore, there is a vertex
$P^{\ell\ell'} \in  \cP^{\ell\ell'}$ such that
\begin{equation}\label{eq:P2}
 P^{\ell\ell'}Q^{k\ell}Q^{k\ell'},\, 
 P^{\ell\ell'}Q^{k'\ell}Q^{k'\ell'},\, 
 P^{\ell\ell'}S^{\ell m}S^{\ell'm}\in E(\cA)\,. 
\end{equation}

Finally, since $\cQ(K, L)\cR(K, M)\cS(L, M)$ is inhabited,
we have 
\begin{equation}\label{eq:P3}
 Q^{k\ell}R^{k m}S^{\ell m},\,
 Q^{k\ell'}R^{k m}S^{\ell' m},\,
 Q^{k'\ell}R^{k' m}S^{\ell m},\,
 Q^{k'\ell'}R^{k' m}S^{\ell' m} 
 \in E(\cA)\,. 
\end{equation}
Altogether the ten hyperedges provided by~\eqref{eq:P1}\,--\,\eqref{eq:P3}
show that the 
vertices $P^{kk'}$, $P^{\l\l'}$, together with $Q^{k\l}$, $Q^{k\l'}$, 
$Q^{k'\l}$, $Q^{k'\l'}$, $R^{km}$, $R^{k'm}$, and~$S^{\l m}$, $S^{\l'm}$ support 
a $K_5^{(3)}$ on the five indices $k$, $k'$, $\l$, $\l'$, and $m$.
\end{proof}

By means of the product Ramsey theorem (see, e.g.,~\cite{Ramseybook}*{Theorem 5.1.5})
we can move from at most one pair with intersecting links 
(given by Lemma~\ref{lem:greekleters})
to at least two pairs with essentially disjoint links.

\vbox{
\begin{cor}\label{cor:greekleters}
Let $t\in\NN$,~$\delta >0$, let~$\cA$ be a reduced hypergraph with index set~$I$ that does not support $K_5^{(3)}$, and let~$\cQ(K, L)\cR(K, M)\cS(L, M)$ be an inhabited triple of transversals
for sufficiently large disjoint sets~$K$, $L$, $M\subseteq I$.

Then there exist subsets~$K_\star\subseteq K$, $L_\star\subseteq L$, and~$M_\star\subseteq M$ each of size~$t$
such that at most one pair of restricted transversals~$\cQ(K_\star, L_\star)\cR(K_\star, M_\star)$, 
$\cQ(K_\star, L_\star)\cS(L_\star, M_\star)$, $\cR(K_\star, M_\star)\cS(L_\star, M_\star)$ has 
$\delta$-intersecting links and all other pairs have~$\delta$-disjoint links.
\end{cor}}

\begin{proof}
Define a $2$-colouring on the triples~$(kk',\ell,m)\in K^{(2)}\times L\times M$ depending on whether~$N(Q^{k\ell}, Q^{k'\ell})\cap N(R^{k m}, R^{k' m}) \geq \delta \vert \cP^{k k'}\vert$ or not. 
Since~$K$, $L$, and $M$ are large enough, we can deduce from the product Ramsey theorem  that there exist large subsets~$K_1\subseteq K$, $L_1\subseteq L$, and~$M_1\subseteq M$ for which the pair of restricted transversals~$\cQ(K_1, L_1)\cR(K_1, M_1)$ has~$\delta$-intersecting
or~$\delta$-disjoint links. 

We can repeat this argument and consider the triples in~$L_1^{(2)}\times K_1 \times M_1$ 
to obtain subsets~$K_2\subseteq K_1,L_2\subseteq L_1$, and~$M_2\subseteq M_1$ such that 
the pair~$\cQ(K_2, L_2)\cS(L_2, M_2)$ 
has {$\delta$-intersecting} or~$\delta$-disjoint links. 
Observe that these properties are preserved under taking subsets of indices and, hence, 
also the pair~$\cQ(K_2, L_2)\cR(K_2, M_2)$ has~$\delta$-intersecting 
or~$\delta$-disjoint links. 

Repeating the Ramsey argument again yields subsets~$K_\star\subseteq K_2, L_\star\subseteq L_2$, and~$M_\star\subseteq M_2$ such that all pairs of restricted transversals~$\cQ(K_\star, L_\star)$,~$\cR(K_\star, M_\star)$, and~$\cS(L_\star, M_\star)$ have {$\delta$-intersecting} or~$\delta$-disjoint links. 
Since the initial sets~$K$, $L$, and~$M$ are large enough, we argue that~$K_\star$, $L_\star$, and~$M_\star$ can be taken of size at least~$t$.

Finally, applying Lemma~\ref{lem:greekleters} we observe that at most one of those pairs of transversals has 
a~$\delta$-intersecting link, and hence, at least two of them have~$\delta$-disjoint links. 
\end{proof}

Finally, we may combine Corollaries~\ref{cor:lamda} and~\ref{cor:greekleters}. More precisely, after 
an application of Corollary~\ref{cor:greekleters} and three consecutive applications of 
Corollary~\ref{cor:lamda} we arrive at the following statement.

\begin{cor}\label{cor:cleaning}
    Let~$t\in \NN$, $\delta$, $\mu$, $d>0$, let~$\cA$ be a~$\left(d, \ee\right)$-dense 
    reduced hypergraph with index set~$I$ that does not support a~$K_5^{(3)}$,
    and for sufficiently large disjoint sets~$K$, $L$, $M\subseteq I$ let~$\cQ(K,L)\cR(K,M)\cS(L,M)$ be an inhabited triple of transversals.
    
    There exist subsets~$K_\star\subseteq K$,~$\Ls\subseteq L$, and~$M_\star\subseteq M$ of 
    size at least~$t$ such that 
    \begin{enumerate}[label=\rmlabel]
    \item\label{it:cleaning1} at most one pair
    	$\cQ(K_\star, L_\star)\cR(K_\star, M_\star)$, $\cQ(K_\star, L_\star)\cS(L_\star, M_\star)$, 
		$\cR(K_\star, M_\star)\cS(L_\star, M_\star)$ of restricted transversals has~$\delta$-intersecting links
		and all other pairs have~$\delta$-disjoint links
    \item\label{it:cleaning2}  and for every $k\in K_{\star}$, $\l\in L_{\star}$, and $m\in M_{\star}$
		the links $\Lambda(\cQ, K_\star,\l)$, $\Lambda(\cQ, L_\star, k)$, $\Lambda(\cR, K_\star, m)$,
    	$\Lambda(\cR, M_\star, k)$, $\Lambda(\cS, L_\star, m)$, and~$\Lambda(\cS, M_\star,\l)$
    	are $(\mu,d)$-holes.\qed
    \end{enumerate}
\end{cor}

\subsection{Equivalent holes}\label{Equivalent holes}
Roughly speaking, in the next step of the proof of Proposition~\ref{prop:reduction} we show that for wicked reduced hypergraphs (see Definition~\ref{def:wicked}),
the set of holes with width bigger than $1/3$ splits into only two classes defined by $\delta$-intersections.
For that we generalise the notion of being $\delta$-intersecting from links to holes.
\begin{definition}\sl
	Given a reduced hypergraph~$\cA$ with index set $I$, a subset $J\subseteq I$,
	and $\mu$, $\delta>0$,
	we say two $\mu$-holes $\Phi$ and $\Psi$ on $J$ are \emph{$\delta$-intersecting}
	if 
	\begin{equation}\label{eq:intersecthole}
		\big|\Phi^{ij}\cap\Psi^{ij}\big|\ge\delta\big|\cP^{ij}\big|
	\end{equation}
	for all $ij\in J^{(2)}$. If, on the other hand,~\eqref{eq:intersecthole} fails for all $ij\in J^{(2)}$, 
	then we say $\Phi$ and $\Psi$ are \emph{$\delta$-disjoint}.
\end{definition}

For  $\mu>0$ and $\delta\in (0,1]$ the notion of being $\delta$-intersecting defines a reflexive and symmetric relation on 
the $\mu$-holes on $J$. Perhaps somewhat surprisingly, the next lemma shows that 
this relation is also transitive on holes with width bigger than $1/3$ in wicked
reduced hypergraphs,
if one passes to an appropriate subset of~$J$. This justifies the shorthand notation 
\[
	\Phi\equiv_{\delta,J}\Psi
\] 
for $\delta$-intersecting holes on~$J$. Similarly, $\Phi\not\equiv_{\delta,J}\Psi$ 
will indicate that $\Phi$ and $\Psi$ are $\delta$-disjoint on $J$. 
Notice that this statement is stronger than the mere negation 
of $\Phi\equiv_{\delta,J}\Psi$.

\begin{lemma}[Transitivity lemma]\label{lem:transitivity}
	For every~$\eps>0$ there exists~$\mu >0$ such that for every~$t\in\NN$ the 
	following holds. 
	Suppose~$\cA$ is an $\eps$-wicked reduced hypergraph with index set $I$ and 
	for sufficiently large $J\subseteq I$ we are given~$(\mu,1/3+\eps)$-holes~$\Phi$, 
	$\Psi$, and $\Omega$ on $J$. If
	\[
		\Phi\equiv_{\eps,J}\Psi
		\qqand 
		\Psi\equiv_{\eps,J}\Omega\,,
	\] 
	then there is a subset~$\Js\subseteq J$ of size at least~$t$ 
	such that~$\Phi\equiv_{\eps,\Js}\Omega$.
\end{lemma}

\begin{proof}
	Given $\eps>0$ we fix auxiliary integers $t_1$, $t_2$, $t_3$, and $\mu>0$ satisfying 
	the hierarchy
	\[
	    \eps^{-1}\ll t_3\ll t_2\ll t_1,\,\mu^{-1}\,.
	\]
	Let $t\in \NN$ and  let~$\cA$ be an  $\eps$-wicked reduced hypergraph
	with index set~$I$ and for sufficiently large $J\subseteq I$ 
	let~$\Phi$, $\Psi$, and $\Omega$ be~$(\mu,1/3+\eps)$-holes on $J$ such that 
	the pairs $\Phi\Psi$ and $\Psi\Omega$ are $\eps$-intersecting. In particular, we may assume that 
	the Ramsey theoretic statement $|J|\lra (t_1, t)_2^2$ holds, which asserts that every $2$-colouring of 
	$J^{(2)}$ yields a subset of $t_1$ vertices with all its pairs having the first colour or a subset of size 
	$t$ with all its pairs displaying the second colour.
	 
    Consider an auxiliary~$2$-colouring of the pairs~$ij\in J^{(2)}$ depending on whether
    \begin{align}\label{eq:nonequivholes}
        \vert \Phi^{ij}\cap \Omega^{ij}\vert 
        < 
        \eps\vert\cP^{ij}\vert
    \end{align}
    or not.
    Since~$|J|\lra (t_1, t)_2^2$, there either exists the desired set $J_\star$,
    or there is a subset~$J_1\subseteq J$ of size $t_1$
    such that~\eqref{eq:nonequivholes} holds for every~$ij\in J_1^{(2)}$.  
    So it suffices to show that the second possibility contradicts the wickedness 
    of~$\cA$. 
    
    First we note that for all $i<j<k$ from $J_1$ and every $P^{ij}\in\cP^{ij}$ 
    and $P^{jk}\in\cP^{jk}$ the $(1/3+\eps,\ee)$-density of $\cA$ and the given 
    width of the holes $\Phi$ and $\Omega$ together with~\eqref{eq:nonequivholes} imply
    \begin{align}
    	\big|N(P^{ij},P^{jk})\cap ( \Phi^{ik}\cup \Omega^{ik})\big|
		&\ge
		\big|N(P^{ij},P^{jk})\big|
		+
		\big|\Phi^{ik}\big|
		+
		\big|\Omega^{ik}\big|
		-
		\big|\cP^{ik}\big|
		-
		\big|\Phi^{ik}\cap \Omega^{ik}\big|\nonumber\\
		&\ge
		2\eps\big|\cP^{ik}\big|\,.\label{eq:oldclaim}
    \end{align}
    
    We define the reduced subhypergraph $\cA_1 \subseteq \cA$ with index set $J_1$,
    vertex classes $\cP^{ij}$ inherited from $\cA$, and constituents
    \[
    	\cA_1^{ijk}
		= 
		\cA^{ijk}[\Phi^{ij}\cap \Psi^{ij}, \Phi^{ik}\cup\Omega^{ik}, 
			\Psi^{jk}\cap \Omega^{jk}]\,.
	\]
	Since the pairs $\Phi\Psi$ and $\Psi\Omega$ are $\eps$-intersecting, 
    we infer from~\eqref{eq:oldclaim}  for all $i < j < k$ in~$J_1$ that
    \begin{align*}
        e(\cA_1^{ijk})      
        = 
        \sum_{\substack{
        	P^{ij} \in \Phi^{ij} \cap \Psi^{ij}\\
        	P^{jk} \in \Psi^{jk} \cap \Omega^{jk}}} 
		\big|N_{\cA}(P^{ij},P^{jk})\cap ( \Phi^{ik}\cup \Omega^{ik})\big|
        \ge 
        2 \varepsilon^3 \vert \cP^{ij}\vert \vert \cP^{ik}\vert \vert  \cP^{jk}\vert
    \end{align*}
    and, hence, $\cA_1$ is $(2 \varepsilon^3,\vvv)$-dense.
        
    We consider the $\eps$-exceptional left and right cherries (see Definition~\ref{def:exceptional-cherries})
    of the holes~$\Phi$,~$\Psi$, and~$\Omega$ (restricted to $J_1$) and	
    for every $i < j < k$ in $J_1$ we set
	\begin{equation*}
		\ccL^{ijk}=\ccL^{ijk}(\Psi,\eps)\cup\ccL^{ijk}(\Omega,\eps)
	    \qand 
	    \ccR^{ijk}=\ccR^{ijk}(\Phi,\eps)\cup\ccR^{ijk}(\Psi,\eps)\,.
	\end{equation*}
    We infer from~\eqref{eq:except-cherries} that
   	\begin{equation*}    \vert \ccL^{ijk}\vert 
    \leq 
    \frac{2\mu}{\eps} \vert \cP^{ij}\vert \vert  \cP^{ik}\vert
    \qand
    \vert \ccR^{ijk}\vert 
    \leq 
    \frac{2\mu}{\eps} \vert \cP^{ik}\vert \vert  \cP^{jk}\vert
    \,.
    \end{equation*}

    By the choice of $\mu$ we can apply Lemma~\ref{lem:inhabitedtransversals} to~$\cA_1$
    with $t_2$, $2\eps^3$, and $\tfrac{2\mu}{\eps}$ in place of~$t$,~$\mu$ and~$\mu'$.
    This yields a set $J_2 \subseteq J_1$ of size $t_2$ and an inhabited triple 
    of transversals $\cQ(J_2)\cR(J_2)\cS(J_2)$ avoiding the exceptional 
    cherries from~$\ccL^{ijk}$ and~$\ccR^{ijk}$ for every $ijk\in J_2^{(3)}$. 
    In particular, for all $i < j < k$ in $J_2$ we have
    \begin{equation}
        Q^{ij} R^{ik} S^{jk} \in E(\cA_1^{ijk})= E\big(\cA^{ijk}[
			\Phi^{ij}\cap \Psi^{ij}, \Phi^{ik}\cup\Omega^{ik}, \Psi^{jk}\cap \Omega^{jk}
			]\big)\,.\label{eq:eqiv-inhabited}
    \end{equation}
    Recall that the index set $I$ is ordered and for two subsets $A$, $B\subseteq I$ we write $A<B$ to signify that $\max A < \min B$.
    We fix disjoint subsets~$K' < L < M'$ of $J_2$, where $K'$ and $M'$ have size~$\lfloor t_2 /3\rfloor$ and~$L$ has size~$t_3$.
    
    Note that by definition $\cR(K',M') \subseteq \Phi \cup \Omega$. Due to the product 
    Ramsey theorem, applied with the set of colours $\{\Phi, \Omega \}$, 
    this leads to sets $K \subseteq K'$ and $M \subseteq M'$ of size $t_3$ 
    and to a hole $\Pi \in \{ \Phi, \Omega \}$ such that
    \begin{equation}\label{eq:wlogPi}
    	R^{km} \in \Pi^{km}\ \text{for every $k \in K$, and $m \in M$.}
	\end{equation}
    
    Owing to~\eqref{eq:eqiv-inhabited} the restricted transversals 
    $\cQ(K,L)$, $\cR(K,M)$, and $\cS(L,M)$ form an inhabited triple in $\cA$.
    We derive a contradiction by Lemma~\ref{lem:greekleters} and for that 
    we shall show that two of the pairs $\cQ(K,L)\cR(K,M)$, $\cQ(K,L)\cS(L,M)$, and $\cR(K,M)\cS(L,M)$ have $\eps$-intersecting
    links.
    
    First, we recall that, independent of the chosen $\Pi$, the pair $\cQ(K,L)\cS(L,M)$ consists of transversals 
    inside the hole $\Psi$ and both avoid 
    the exceptional  left and right cherries from~$\Psi$.
    Hence, for all
    $k\in K$, $\l\l'\in L^{(2)}$, and $m\in M$ we have
    \begin{equation*}
        \big|N
    	_{\cA}(Q^{k\l},Q^{k\l'})\cap\Psi^{\l\l'}\big|<\eps \big|\cP^{\l\l'}\big|
		\qand 
		\big|N_{\cA}(S^{\ell m},S^{\l'm})\cap\Psi^{\l\l'}\big|< \eps \big|\cP^{\l\l'}\big|\,.
    \end{equation*}
    Consequently, the $(1/3+\eps,\ee)$-density of $\cA$ and the width of $\Psi$ imply
    \begin{equation*}
    	\big|N_{\cA}(Q^{k\l},Q^{k\l'})\cap N_{\cA}(S^{\l m},S^{\l'm})\big|
		> 
		\eps \big|\cP^{\l\l'}\big|
    \end{equation*}
    for every $k\in K$, $\l\l'\in L^{(2)}$, and $m\in M$, i.e.,  the pair $\cQ(K,L)\cS(L,M)$
    has $\eps$-intersecting links.
    
    If $\Pi = \Phi$, then $\cQ(K,L)$ and $\cR(K,M)$ are both transversals in $\Phi$ (see~\eqref{eq:wlogPi}) and both~$\cQ$ and $\cR$ avoid the exceptional right cherries of $\Phi$.
    As before, this implies that the pair $\cQ(K,L)\cR(K,M)$ has~$\eps$-intersecting links.
    So Lemma~\ref{lem:greekleters} tells us that~$\cA$ supports a~$K_5^{(3)}$, contrary 
    to the wickedness of $\cA$.
    
     Analogously, if $\Pi = \Omega$, then $\cR(K,M)$ and $\cS(L,M)$ are both transversals in $\Omega$
      and, since both~$\cR$ and $\cS$ avoid the exceptional left cherries of $\Omega$, the pair of transversals has~$\eps$-intersecting links,
      which leads to the same contradiction.
\end{proof}

Another application of Ramsey's theorem leads to the following corollary.

\vbox{\begin{cor}\label{cor:equivalent}
	For every~$\eps\in(0,1]$ there exists~$\mu >0$ such that for all integers~$t$, $r\geq 2$ the following holds. 
	Suppose~$\cA$ is an $\eps$-wicked reduced hypergraph with index set $I$ and for sufficiently large 
	$J\subseteq I$ we are given~$(\mu,1/3+\eps)$-holes~$\Phi_1,\dots,\Phi_r$ on $J$.
	
	Then there exists a subset~$\Js\subseteq J$ of size~$t$ such that~
	\begin{enumerate}[label=\rmlabel]
	\item\label{it:equiv1} for all $\rho$, $\rho'\in[r]$ the holes $\Phi_{\rho}$ and $\Phi_{\rho'}$
		are either $\eps$-intersecting or $\eps$-disjoint on~$\Js$
	\item\label{it:equiv2}
	and $\equiv_{\eps,\Js}$ is an equivalence relation 
	on~$\{\Phi_1,\dots,\Phi_r\}$ with at most two equivalence classes.
	\end{enumerate}
\end{cor}}

\begin{proof}
	For $\eps\in(0,1]$ let $\mu>0$ be given by Lemma~\ref{lem:transitivity}. For fixed $t$, $r\geq 2$  let $t'\geq t$ 
	be sufficiently large for an application of Lemma~\ref{lem:transitivity} 
	with $\eps$, $\mu$, and with $2$ in place of $t$.
	
	For a given $\eps$-wicked reduced hypergraph~$\cA$ and $(\mu,1/3+\eps)$-holes~$\Phi_1,\dots,\Phi_r$ 
	we impose 
	that the size of $J$ is at least the $2^{\binom{r}{2}}$-colour Ramsey number for graph cliques on $t'$ vertices, i.e.,
	\begin{equation}\label{eq:quivalence-pf}
		|J|\lra(t')^2_{|\Xi|}
		\quad
		\text{for}
		\quad
		\Xi
		=\big\{\xi=(\xi_{\rho\rho'})_{\rho\rho'\in[r]^{(2)}}
			\colon \xi_{\rho\rho'}\in\{0,1\}\ \text{for}\ \rho\rho'\in[r]^{(2)}
			\big\}\,.
	\end{equation}
	We assign to a pair $ij\in J^{(2)}$ 
	the colour $\xi=(\xi_{\rho\rho'})_{\rho\rho'\in[r]^{(2)}}$ with 
	$\xi_{\rho\rho'}=1$ signifying
	\[
		\big|\Phi^{ij}_\rho\cap\Phi^{ij}_{\rho'}\big|
		\ge
		\eps\big|\cP^{ij}\big|
	\]
	and $\xi_{\rho\rho'}=0$ otherwise. Owing to~\eqref{eq:quivalence-pf} there exists a subset $\Js\subseteq J$ of size at least 
	$t'\geq t$ and a colour $\xi^\star=(\xi^{\star}_{\rho\rho'})_{\rho\rho'\in [r]^{(2)}}$ 
	such that all pairs of $\Js$ were assigned 
	$\xi^{\star}$. Note that assertion~\ref{it:equiv1} follows directly from the definition of the colouring, i.e.,
	$\Phi_{\rho}$ and $\Phi_{\rho'}$ are $\eps$-intersecting on~$\Js$ if $\xi^{\star}_{\rho\rho'}=1$ 
	and $\eps$-disjoint otherwise.
	
	Obviously the relation $\equiv_{\eps,\Js}$ is reflexive and symmetric. Moreover, our choice of $t'$
	allows us to invoke Lemma~\ref{lem:transitivity} and the transitivity follows from the definition 
	of the colouring. Since all holes have width at least $1/3+\eps$, at least two among any choice of 
	three holes must share at least $\eps|\cP^{ij}|$ vertices in $\cP^{ij}$ for any $ij\in\Js^{(2)}$ and, hence,
	$\equiv_{\eps,\Js}$ has at most two equivalence classes.
\end{proof}

It will later be important that, under sufficiently general circumstances, there 
really are two distinct equivalence classes. 

\begin{lemma}\label{lem:1247}
	Given $t\in \NN$ and $\eps, \mu>0$ let $I$ be a sufficiently large set of indices.
	For every $\eps$-wicked reduced hypergraph $\cA$ with index set $I$ there 
	exist a set $J\subseteq I$ of size $t$ and two $\eps$-disjoint 
	$(\mu, 1/3+\eps)$-holes on $J$.
\end{lemma}

\begin{proof}
We may assume that we have an integer $t'$ fitting into the 
hierarchy 
\[
	|I|\gg t'\gg t, \mu^{-1}\,.
\]
Since $\cA$ is, in particular, 
$(1/3+\eps,\vvv)$-dense, Theorem~\ref{lem:nullpaper} applied with 
$3t'$, $1/3+\eps$ here in place of $t, \mu$ there yields a set $I'\subseteq I$ 
of size~$3t'$ and an inhabited triple of transversals~$\cQ(I')\cR(I')\cS(I')$.

Fix an arbitrary partition $I'=K'\dcup L'\dcup M'$ such that $|K'|=|L'|=|M'|=t'$.
Now we apply Corollary~\ref{cor:cleaning} with $\eps$, $1/3+\eps$ here in 
place of $\delta$, $d$ there to the inhabited triple of restricted
transversals $\cQ(K',L')\cR(K',M')\cS(L',M')$. This yields
subsets~$K\subseteq K'$,~$L\subseteq L'$, and~$M\subseteq M'$ of size~$t$ 
satisfying properties~\ref{it:cleaning1} and~\ref{it:cleaning2} of the corollary.

By~\ref{it:cleaning1} we may assume without loss of generality 
that the pair~$\cQ(K,L)\cR(K,M)$ has {$\eps$-disjoint} links. 
Thus, fixing~$\ell\in L$ and~$m\in M$ arbitrarily we obtain
the desired~$\eps$-disjoint $(\mu, 1/3+\eps)$-holes 
$\Lambda(\cQ,K,\ell)$ and $\Lambda(\cR,K,m)$ on $J=K$.
\end{proof}

\subsection{Unions of equivalent holes}\label{sec:union}

We proceed with the union lemma, which roughly speaking asserts that unions of 
equivalent holes are holes. As usual, the precise statement involves a considerable 
loss in the number of relevant indices. Moreover, if want such a union $\Phi\cup\Psi$ to be 
a $\mu$-hole, we need to assume that $\Phi$ and $\Psi$ themselves are $\nu$-holes for 
some very small~$\nu\ll\mu$.

\begin{lemma}[Union lemma]\label{lem:union}
For every~$\mu$,~$\eps>0$ there exists~$\nu>0$ such that for every~$t\in \NN$ the following holds. 
Suppose~$\cA$ is an~$\eps$-wicked reduced hypergraph with index set~$I$ and for a sufficiently large subset~$J\subseteq I$ we are given two $(\nu, 1/3+\eps)$-holes~$\Phi$ and~$\Psi$ on~$J$ 
such that~$\Phi \equiv_{\eps, J} \Psi$. 
Then, there exists a subset~$J_\star\subseteq J$ of size at least~$t$ such that~$\Phi\cup \Psi$ is a~$(\mu,1/3+\eps)$-hole on~$J_\star$.
\end{lemma}

\begin{proof}
	As decreasing $\mu$ makes the lemma stronger, we may assume that $\mu\ll\eps$. 
	Now we take auxiliary integers $t_1$, $t_2$, $t_3$, and $t_4$ and a positive 
	real $\nu$
	fitting into the hierarchy 
\[
		\eps\gg t_4^{-1}\gg t_3^{-1}\gg t_2^{-1}, \nu\gg t_1^{-1}\,.
	\]
More precisely we assume that  
	\begin{enumerate}[label=\nlabel]
	\item\label{eq:union_t4} $t_4$ is so large that the conclusion of 
		Corollary~\ref{cor:equivalent} holds for $\eps$, $\mu$, and for
		$2$, $4$, $t_4$
		here in place of $t$, $r$, $|J|$ there;
	\item\label{eq:union_t3} 
	$t_3$ is so large that the conclusion of Corollary~\ref{cor:cleaning} holds 
	for $t_4$, $\eps$, $\mu$, $1/3+\eps$, $t_3$
	here in place of 
	$t$, $\delta$, $\mu$, $d$, $\min\{|K|, |L|, |M|\}$
	there;
	\item\label{eq:union_t2} $t_2$ is so large and $\nu\le \mu$ is so small 
	that the conclusion 
	of Lemma~\ref{lem:inhabitedtransversals}
	holds for $3t_3$, $\mu/8$, $2\nu/\eps$, $t_2$       
	here in place of 
	$t$, $\mu$, $\mu'$, $|J|$
	there;
	\item\label{eq:union_t1}
		and $t_1\lra (t_2)^3_8$.
	\end{enumerate}	 
	
	Finally, given $t\in\NN$ we suppose that $J\subseteq I$ is sufficiently large so that 
	\[
		|J|\lra (t_1, t)^3_2\,.
	\]	
	For $(\nu, 1/3+\eps)$-holes~$\Phi$ and~$\Psi$ on $J$ let 
	\[
		\ccL=\ccL(\Phi,\eps)\cup\ccL(\Psi,\eps)
		\qand
		\ccR=\ccR(\Phi,\eps)\cup\ccR(\Psi,\eps)
	\]
	be their $\eps$-exceptional left and right cherries. For later reference we recall that~\eqref{eq:except-cherries}
	yields
	\begin{equation}\label{eq:boundcherriesunion}
    	\vert \ccL^{ijk}\vert 
    	\leq 
    	\frac{2\nu}{\eps} \vert \cP^{ij}\vert \vert  \cP^{ik}\vert
    	\qand
    	\vert \ccR^{ijk}\vert 
    	\leq 
    	\frac{2\nu}{\eps} \vert \cP^{ik}\vert \vert  \cP^{jk}\vert
    \,.
    \end{equation}
    
    We  begin with an application of Ramsey's theorem for hypergraphs and 
    consider a $2$-colouring of the triples $ijk\in J^{(3)}$ depending on whether
    \begin{align}\label{eq:union}
        e(\Phi^{ij}\cup \Psi^{ij}, \Phi^{ik}\cup \Psi^{ik}, \Phi^{jk}\cup \Psi^{jk}) 
        >
        \mu \vert\cP^{ij}\vert \vert\cP^{ik}\vert\vert\cP^{jk}\vert
    \end{align}
    or not. Owing to the size of $J$, there either exists the desired set $J_\star$,
    or there is a subset~$J_1\subseteq J$ of size $t_1$ 
    such that~\eqref{eq:union} holds for all~$ijk\in J^{(3)}_1$. 
    We shall show that the second case leads to a contradiction. 
    
    First we observe that for every~$ijk\in J_1^{(3)}$ inequality~\eqref{eq:union} 			implies that for at least one 
		of the eight possible triples~$(\Pi_1,\Pi_2, \Pi_3) \in \{\Phi,\Psi\}^3$ we have 
    \begin{align}\label{eq:crossdensepi}
        e(\Pi_1^{ij}, \Pi_2^{ik}, \Pi_3^{jk}) 
        >
        \frac{\mu}{8} \vert\cP^{ij}\vert \vert\cP^{ik}\vert\vert\cP^{jk}\vert\,.
    \end{align}
     (Actually, since $\Phi$ and $\Psi$ are $\nu$-holes and~$\nu\leq\mu/8$, 
    inequality~\eqref{eq:crossdensepi} can neither hold for~$e(\Phi^{ij},\Phi^{ik},\Phi^{jk})$ nor for 
    $e(\Psi^{ij},\Psi^{ik},\Psi^{jk})$, but we shall not use this here.)
    Thus, there exists an $8$-colouring of $J^{(3)}_1$ such that if the colour 
    of a triple $ijk\in J^{(3)}_1$ is~$(\Pi_1,\Pi_2, \Pi_3) \in \{\Phi,\Psi\}^3$,  
    then this indicates the validity of~\eqref{eq:crossdensepi}
    for this triple of holes. 
    In view of~\ref{eq:union_t1} there exists a subset~$J_2\subseteq J_1$ of size~$t_2$ 
    and there is a fixed colour $(\Pi_1,\Pi_2, \Pi_3) \in \{\Phi,\Psi\}^3$ such that
    inequality~\eqref{eq:crossdensepi} holds for every~$ijk\in J^{(3)}_2$.

Now the reduced subhypergraph $\cA_2\subseteq\cA$ with index set $J_2$, 
vertex classes inherited from~$\cA$, and constituents 
\begin{equation}\label{eq:union-trans}
	\cA^{ijk}_2
	=
	\cA^{ijk}[\Pi_1^{ij}, \Pi_2^{ik}, \Pi_3^{jk}]
\end{equation} 
for all $i<j<k$ in $J_2$ is~$(\mu/8,\vvv)$-dense.
Owing to~\eqref{eq:boundcherriesunion} and our choice of $t_2$ and $\nu$ 
in~\ref{eq:union_t2}, Lemma~\ref{lem:inhabitedtransversals}
ensures that there is a subset~$J_3\subseteq J_2$ of size $3t_3$ 
and there is an inhabited triple of transversals 
$\cQ(J_3)\cR(J_3)\cS(J_3)$ where each transversal avoids the sets of exceptional 
left and right cherries~$\ccL$ and~$\ccR$ of $\Phi$ and $\Psi$. 

Since $\cQ(J_3)\cR(J_3)\cS(J_3)$ is an inhabited triple,
we have $Q^{ij}R^{ik}S^{jk}\in E(\cA_2)$ for every $i<j<k$ in $J_3$ 
and, therefore,~\eqref{eq:union-trans} implies 
\begin{equation}\label{eq:union-trans2}
	Q^{ij}\in \Pi_1\,,\quad
	R^{ik}\in \Pi_2\,,\qand
	S^{jk}\in \Pi_3
\end{equation}
for all $i<j<k$ in $J_3$.

    Fix disjoint subsets of indices~$K_3<L_3<M_3$ of $J_3$ each of size~$t_3$. 
    Clearly, the triple of restricted 
    transversals~$\cQ(K_3,L_3)\cR(K_3,M_3)\cS(L_3,M_3)$ 
    is still inhabited. 
    Therefore, the choice of $t_3$ in~\ref{eq:union_t3} 
    allows an application of Corollary~\ref{cor:cleaning}, which 
    yields subsets~$K_4\subseteq K_3$, $L_4\subseteq L_3$, and~$M_4\subseteq M_3$
     each of size $t_4$
    satisfying properties~\ref{it:cleaning1} and~\ref{it:cleaning2} 
    of Corollary~\ref{cor:cleaning}.
    
    Next we shall show that all three pairs of restricted transversals $\cQ(K_4, L_4)\cR(K_4,M_4)$, 
    $\cQ(K_4, L_4)\cS(L_4,M_4)$, and $\cR(K_4, M_4)\cS(L_4,M_4)$ have $\eps$-intersecting links. 
    However, this contradicts property~\ref{it:cleaning1} of 	
    Corollary~\ref{cor:cleaning}, which allows only one pair of transversals with 
    $\eps$-intersecting links and this contradiction concludes the proof of Lemma~\ref{lem:union}.
    Below we show that the pair $\cQ(K_4, L_4)\cR(K_4,M_4)$ has $\eps$-intersecting 
    links. The proof 
    for the other pairs follows essentially verbatim the same lines.
    
    Fix some~$\l\in L_4$ and~$m\in M_4$. Property~\ref{it:cleaning2} 
    of Corollary~\ref{cor:cleaning} 
    tells us that $\Lambda(\cQ, K_4, \l)$ and  $\Lambda(\cR, K_4, m)$ 
    are $(\mu,1/3+\eps)$-holes on~$K_4$. 
    Moreover, since $\nu\leq\mu$, also $\Phi$ and $\Psi$ 
    are $(\mu,1/3+\eps)$-holes on~$K_4$
	and, therefore, the choice of $t_4$ in~\ref{eq:union_t4} and an application of 
	Corollary~\ref{cor:equivalent} yield a subset~$\Ks\subseteq K_4$ of size two 
	such that~$\equiv=\equiv_{\eps, \Ks}$ is an equivalence relation with at most 
	two equivalence classes
	on the $\mu$-holes
	\[
		\Lambda(\cQ, \Ks, \l)\,,\quad \Lambda(\cR, \Ks, m)\,,\quad \Pi_1\,,\qand \Pi_2\,. 
    \]
    
    In view of~\eqref{eq:union-trans2} we have $\cQ(\Ks, L_4)\subseteq \Pi_1$
    and~$\cR(\Ks, M_4)\subseteq \Pi_2$ and, since $\cQ$ and $\cR$ avoid the exceptional cherries from~$\ccL$ 
    and~$\ccR$, we infer
    \begin{align*}
          \vert N(Q^{k\ell}, Q^{k'\ell})\cap \Pi_1^{kk'}\vert <\eps \vert \cP^{kk'}\vert 
          \qqand
          \vert N(R^{km}, R^{k'm})\cap \Pi_2^{kk'}\vert <\eps \vert \cP^{kk'}\vert\,,
    \end{align*}
    where $k$ and $k'$ denote the two elements of $\Ks$. Consequently, 
    \begin{align*}
        \Pi_1 \not\equiv\Lambda(\cQ, \Ks,\l) 
        \qand 
        \Pi_2 \not\equiv \Lambda(\cR, \Ks,m)\,.                            
    \end{align*}
    The assumption $\Phi\equiv \Psi$ of the lemma yields $\Pi_1\equiv \Pi_2$ and $\equiv$
    has at most two equivalence classes, we thus arrive at 
    \[
    	\Lambda(\cQ, \Ks,\l)\equiv\Lambda(\cR, \Ks, m)\,.
	\]
	
	In other words we 
	have $|N(Q^{k\ell}, Q^{k'\ell})\cap N(R^{km}, R^{k'm})|\ge \eps|\cP^{kk'}|$,
	which excludes the possibility that the pair of transversals 
	$\cQ(K_4, L_4)\cR(K_4,M_4)$ has~$\eps$-disjoint links. 
	So by property~\ref{it:cleaning1} of Corollary~\ref{cor:cleaning} it follows 
	that this pair has~$\eps$-intersecting links, as desired.
\end{proof}

For later reference we now state a corollary that follows from Corollary~\ref{cor:equivalent} and Lemma~\ref{lem:union}. 

\begin{cor}\label{cor:adhoc}
For every~$\mu$,~$\eps>0$ there exists~$\nu>0$ such that for every~$t\in \NN$ the following holds. 
Suppose~$\cA$ is an~$\eps$-wicked reduced hypergraph with index set~$I$ and for a sufficiently large subset~$J\subseteq I$ we are given three~$(\nu, 1/3+\eps)$-holes~$\Phi$,~$\Psi$, and~$\Omega$ on $J$ such that~$\Phi$ and~$\Psi$ are~$\eps$-disjoint.

Then, there exists a subset~$J_\star\subseteq J$ of size at least~$t$ such that \begin{enumerate}[label=\Alabel]
    \item \label{it:adhocphi} either $\Phi\cup\Omega$ is a~$(\mu,1/3+\eps)$-hole~$\eps$-disjoint with~$\Psi$
    \item \label{it:adhocpsi}or $\Psi\cup\Omega$ is a~$(\mu,1/3+\eps)$-hole~$\eps$-disjoint with~$\Phi$.
\end{enumerate}
\end{cor}

\begin{proof}
Again we may assume that $\mu\ll\eps$. Take appropriate constants 
\[
	\nu\ll\mu
	\qand 
	t_1\gg t_2\gg t, \nu^{-1}\,, 
\]
and assume that $|J|\gg t_1$. 

Due to Corollary~\ref{cor:equivalent} there is a subset~$J_1\subseteq J$ 
of size~$t_1$ such that~$\equiv_{\eps, J_1}$ is an equivalence relation 
with at most two equivalence classes on~$\{\Phi,\Psi,\Omega\}$. 
By hypothesis the holes~$\Phi$ and~$\Psi$ are in different classes
and thus we may assume without loss of generality that 
\[
    \Omega\equiv_{\eps,J_1}\Phi
    \qqand
    \Omega\not\equiv_{\eps,J_1}\Psi\,.
\]

An application of Lemma~\ref{lem:union} yields the existence of a 
subset~$J_2\subseteq J_1$ of size~$t_2$ on which
\[
	\Phi\cup \Omega\text{ is a }(\mu,1/3+\eps)\text{-hole.}
\]

Now a second application of Corollary~\ref{cor:equivalent} 
leads to a $t$-element subset $J_\star\subseteq J_2$ such 
that~$\equiv_{\eps,J_\star}$ is an equivalence relation with at most 
two equivalence classes on $\{\Phi, \Phi\cup\Omega, \Psi\}$.
Since $\Phi\cup\Omega\equiv_{\eps,J_\star}\Phi\not\equiv_{\eps,J_\star}\Psi$,
we have $\Phi\cup\Omega\not\equiv_{\eps,J_\star}\Psi$. Altogether both parts of~\ref{it:adhocphi} hold. 
\end{proof}

\subsection{Holes derived from two transversals} \label{sec:hdtt}

Before we can make further progress, we need to analyse holes generated 
by two transversals. Given two transversals $\cQ(J)$ and $\cR(J)$ 
in a wicked reduced hypergraph $\cA$, we wonder whether for fixed $i\in J$
the sets $\Omega^{jk}_i\subseteq \cP^{jk}$ defined by $\Omega^{jk}_i=N(Q^{ij}, R^{ik})$
form a hole. There are several possible cases depending on how $i$, $j$, $k$ are ordered,
and in the lemma that follows we focus on the case $i<j<k$. 
It turns out that if the links of $\cQ$ and $\cR$ satisfy a certain equivalence 
condition (see~\eqref{eq:2103} below), then on a large subset of $J$ 
the sets $\Omega^{jk}_i$ form holes.   

\begin{lemma}\label{lem:2148}
	Let $\eps>0$, $\nu>0$, and $t\in \NN$ be given and suppose that~$\cA$ 
	is an $\eps$-wicked reduced hypergraph with index set $I$. 
	If $J\subseteq I$ is sufficiently large and $\cQ(J)$, $\cR(J)$ are 
	two transversals on $J$ such that all $i<j<k<\ell$ from $J$ satisfy 
	\begin{equation}\label{eq:2103}
   	|N(Q^{ik}, Q^{jk})\cap N(R^{i\ell}, R^{j\ell})|\ge \eps |\cP^{ij}|\,,
	\end{equation}
	then there is a set $J_\star\subseteq J$ of size $t$ such that 
	we have  
	\[
  		e\bigl(N(Q^{ij}, R^{ik}), N(Q^{ij}, R^{i\ell}),  N(Q^{ik}, R^{i\ell})\bigr) 
  		\le 
		\nu |\cP^{jk}||\cP^{j\ell}||\cP^{k\ell}| 
	\]
	for all $i<j<k<\ell$ in $J_\star$.
\end{lemma}

\begin{proof}
	Suppose that $|J|\gg t_1\gg t_2\gg t, \eps^{-1}, \nu^{-1}$. 
	Set $\Omega^{jk}_i=N(Q^{ij}, R^{ik})$ for all $i<j<k$ from $J$ and colour the 
	quadruples $i<j<k<\ell$ depending on whether 
	\begin{equation}\label{eq:2030}
		e(\Omega_i^{jk}, \Omega_i^{j\ell}, \Omega_i^{k\ell})
		>
		\nu |\cP^{jk}||\cP^{j\ell}||\cP^{k\ell}|
	\end{equation}
	holds or fails. Due to $|J|\lra (4t_1, t)^4_2$ this either leads to the desired 
	set $J_\star$ of size $t$, or to a set $J_1\subseteq J$ of size $4t_1$ such 
	that~\eqref{eq:2030} holds for all $i<j<k<\ell$ in $J_1$.

	Let $J_1=X_1\dcup K_1\dcup L_1\dcup M_1$ be the (unique) partition of $J_1$ into sets 
	of size $t_1$ satisfying $X_1<K_1<L_1<M_1$. 
	Now for every $x\in X_1$ the reduced subhypergraph $\cA_x$ of $\cA$ with 
	index set $K_1\dcup L_1\dcup M_1$, vertex classes inherited from $\cA$, 
	and constituents 
	$\cA_x^{k\ell m}=\cA^{k\ell m}[\Omega_x^{k\ell},\Omega_x^{km}, \Omega_x^{\ell m}]$
	is~$(\nu,\vvv)$-tridense. Therefore, Lemma~\ref{lem:crossinhabited} 
	applied to $t_2$, $1$, $\nu$ here in place of $t$, $r$, $\mu$ there
	yields subsets $X_2\subseteq X_1$,~$K_2\subseteq K_1$,~$L_2\subseteq L_1$, 
	and~$M_2\subseteq M_1$ of size~$t_2$ and a triple of transversals 
	$\cT(K_2,L_2)\cU(K_2,M_2)\cV(L_2, M_2)$ which is inhabited in every $\cA_x$
	with $x\in X_2$. 

	Owing to the definition of the constituents of these reduced hypergraphs
	this means that for all $(x, k, \ell, m)\in X_2\times K_2\times L_2\times M_2$
	we have 
\[
		T^{k\ell}\in \Omega_x^{k\ell}\,,
		\quad
		U^{km}\in \Omega_x^{km}\,,
		\quad
		V^{\ell m}\in \Omega_x^{\ell m}\,,
		\quad \text{ and } \quad 
		T^{k\ell}U^{km}V^{\ell m}\in E(\cA^{k\ell m})\,. 
	\]
In other words, all four triples of transversals
	\begin{multline*}
		\cQ(X_2,K_2)\cR(X_2,L_2)\cT(K_2,L_2)\,, 
		\qquad 
		\cQ(X_2,K_2)\cR(X_2,M_2)\cU(K_2,M_2)\,, \\
		\cQ(X_2,L_2)\cR(X_2,M_2)\cV(L_2,M_2)\,,
		\qqand
		\cT(K_2,L_2)\cU(K_2,M_2)\cV(L_2, M_2)
	\end{multline*}
	are inhabited in $\cA$. 

	We successively apply Corollary~\ref{cor:cleaning} to these four triples of 
	inhabited transversals with $\eps$, $\nu$, $1/3+\eps$ here in place of 
	$\delta$, $\mu$, $d$ there. Each of these applications shrinks the sets of indices 
	still under consideration and eventually we obtain sets~$X_3$,~$K_3$,~$L_3$, 
	and~$M_3$ of size~$2$, which satisfy~\ref{it:cleaning1} and~\ref{it:cleaning2} 
	of Corollary~\ref{cor:cleaning} for all those four inhabited triples of transversals.
	Let us write $X_3=\{x, x'\}$, $K_3=\{k, k'\}$, $L_3=\{\ell, \ell'\}$, and 
	$M_3=\{m, m'\}$.

	Now our assumption on the transversals $\cQ$ and $\cR$ yields 
	\[
		\vert N(Q^{xk},Q^{x'k})\cap N(R^{x\ell}, R^{x'\ell})\vert 
		\geq 
		\eps\vert\cP^{xx'}\vert
	\]
	and thus the pair~$\cQ(X_3,K_3)\cR(X_3,L_3)$ has $\eps$-intersecting links.
	So by~\ref{it:cleaning1} of Corollary~\ref{cor:cleaning} applied to $\cQ\cR\cT$
	the pairs $\cQ(X_3,K_3)\cT(K_3, L_3)$ and $\cR(X_3,L_3)\cT(K_3, L_3)$ 
	have~$\eps$-disjoint links. Similarly, the pairs $\cQ(X_3,K_3)\cR(X_3,M_3)$
	and $\cQ(X_3,L_3)\cR(X_3,M_3)$ have $\eps$-intersecting links, whereas the
	pairs $\cQ(X_3,K_3)\cU(K_3, M_3)$, $\cR(X_3,M_3)\cU(K_3, M_3)$,
	$\cQ(X_3,L_3)\cV(L_3, M_3)$, and $\cR(X_3,M_3)\cV(L_3, M_3)$ 
	have $\eps$-disjoint links.

	Let us now look at the three subsets 
	$N(Q^{xk}, Q^{xk'})$, $N(T^{k\ell}, T^{k'\ell})$, and $N(U^{km}, U^{k'm})$ 
	of~$\cP^{kk'}$. As $\cA$ is $(1/3+\eps, \ee)$-dense, each of them has 
	at least the size $(1/3+\eps)|\cP^{kk'}|$. Moreover, the fact that  
	$\cQ(X_3,K_3)\cT(K_3, L_3)$ and $\cQ(X_3,K_3)\cU(K_3, M_3)$ have $\eps$-disjoint
	links implies 
\[
		|N(Q^{xk}, Q^{xk'})\cap N(T^{k\ell}, T^{k'\ell})| <\eps|\cP^{kk'}| 
		\qand
		|N(Q^{xk}, Q^{xk'})\cap N(U^{km}, U^{k'm})| <\eps|\cP^{kk'}|\,.
	\]
For all these reasons we have 
	$|N(T^{k\ell}, T^{k'\ell})\cap N(U^{km}, U^{k'm})| \ge \eps|\cP^{kk'}|$
	and, hence, the links of $\cT(K_3, L_3)\cU(K_3,M_3)$ are $\eps$-intersecting.  

	Arguing similarly with the subsets 
	$N(R^{xm}, R^{xm'})$, $N(U^{km}, U^{km'})$, and $N(V^{\ell m}, V^{\ell m'})$
	of~$\cP^{mm'}$ one can show that the pair $\cU(K_3, M_3)\cV(L_3,M_3)$
	has $\eps$-intersecting links as well. 
	Thus the  triple $\cT\cU\cV$ yields two pairs of $\eps$-intersecting links, contradicting 
	Corollary~\ref{cor:cleaning}\,\ref{it:cleaning1}.
\end{proof}

We proceed with a related result that, given two transversals $\cQ(J)$, $\cS(J)$,
addresses holes composed of sets of the form $\Omega^{ij}_x=N(Q^{ix}, S^{xj})$,
where $i<x<j$. The proof is very similar to the previous one, but towards the end 
we shall need an additional argument.  

\vbox{
\begin{lemma}\label{lem:2150}
	Given $\eps>0$, $\nu>0$, and $t\in \NN$ let~$\cA$ 
	be an $\eps$-wicked reduced hypergraph with index set $I$. 
	If $J\subseteq I$ is sufficiently large and $\cQ(J)$, $\cS(J)$ are 
	two transversals on $J$ such that all $i<j<k<\ell$ from $J$ satisfy 
	\begin{equation}\label{eq:2137}
    \vert N(Q^{ij}, Q^{ik})\cap N(S^{j\ell}, S^{k\ell})\vert
    \geq 
    \eps\vert \cP^{jk}\vert
    \qand
    \vert N(S^{ik}, S^{i\ell})\cap N(S^{jk}, S^{j\ell})\vert
    \geq 
    \eps\vert \cP^{k\ell}\vert\,,
	\end{equation}
	then there is a set $J_\star\subseteq J$ of size $t$ such that 
	we have  
	\[
  		e\bigl(N(Q^{ix}, S^{xj}), N(Q^{ix}, S^{xk}),  N(Q^{jy}, S^{yk})\bigr) 
  		\le 
		\nu |\cP^{ij}||\cP^{ik}||\cP^{jk}| 
	\]
	for all $i<x<j<y<k$ in $J_\star$.
\end{lemma}}

\begin{proof}
	Since decreasing $\nu$ makes the statement stronger, we may assume 
	that $\nu\ll\eps$.  
	Suppose that $|J|\gg t_1\gg t_2\gg t_3\gg t, \eps^{-1}, \nu^{-1}$. 
	This time we set $\Omega^{ij}_x=N(Q^{ix}, S^{xj})$ for all $i<x<j$ 
	from $J$ and colour the quintuples $i<x<j<y<k$ depending on whether 
	\begin{equation}\label{eq:2230}
	e(\Omega_x^{ij}, \Omega_x^{ik}, \Omega_y^{jk})
	>
	\nu |\cP^{ij}||\cP^{ik}||\cP^{jk}|
	\end{equation}
	holds or fails. Due to $|J|\lra (5t_1, t)^5_2$ this either leads to the desired 
	set $J_\star$ of size $t$, or to a set $J_1\subseteq J$ of size $5t_1$ such 
	that~\eqref{eq:2230} holds for all $i<x<j<y<k$ in $J_1$.

	Now we partition 
	\[
		J_1=K_1\dcup X_1\dcup L_1\dcup Y_1\dcup M_1
	\]
	into $t_1$-sets ordered by $K_1<X_1<L_1<Y_1<M_1$ and form for every 
	pair $(x, y)\in X_1\times Y_1$ the reduced subhypergraph $\cA_{xy}$ 
	of $\cA$ with index set $K_1\dcup L_1\dcup M_1$, vertex classes inherited 
	from $\cA$, 
	and with constituents 
	$\cA_{xy}^{k\ell m}=\cA^{k\ell m}[\Omega_x^{k\ell},\Omega_x^{km}, \Omega_y^{\ell m}]$.
	As these reduced hypergraphs are~$(\nu,\vvv)$-tridense,  
	Lemma~\ref{lem:crossinhabited} applied to $t_2$, $2$, $\nu$ here in place 
	of $t$, $r$, $\mu$ there yields subsets $K_2\subseteq K_1$, 
	$X_2\subseteq X_1$, $L_2\subseteq L_1$, $Y_2\subseteq Y_1$,
	and~$M_2\subseteq M_1$ of size~$t_2$ and a triple of transversals 
	$\cT(K_2,L_2)\cU(K_2,M_2)\cV(L_2, M_2)$, which is inhabited in every $\cA_{xy}$
	with $x\in X_2$ and $y\in Y_2$. 

	As in the proof of the foregoing lemma one observes that the four triples 
	\begin{multline*}
		\cQ(K_2,X_2)\cT(K_2,L_2)\cS(X_2,L_2)\,, 
		\qquad 
		\cQ(K_2,X_2)\cU(K_2,M_2)\cS(X_2,M_2)\,, \\
		\cQ(L_2,Y_2)\cV(L_2,M_2)\cS(Y_2,M_2)
		\qqand
		\cT(K_2,L_2)\cU(K_2,M_2)\cV(L_2, M_2)
	\end{multline*}
	are inhabited in $\cA$. 

	Again, we apply Corollary~\ref{cor:cleaning} successively to all these triples, 
	this time obtaining sets~$K_3$, $X_3$,~$L_3$,~$Y_3$, and~$M_3$ of 
	size~$t_3$, satisfying~\ref{it:cleaning1} and~\ref{it:cleaning2} of 
	Corollary~\ref{cor:cleaning} for these four triples of transversals. 
	As before the desired contradiction arises from the fact that the pairs 
	$\cT(K_3, L_3)\cU(K_3,M_3)$ and $\cU(K_3, M_3)\cV(L_3,M_3)$ 
	have $\eps$-intersecting links, contrary to 
	Corollary~\ref{cor:cleaning}~\ref{it:cleaning1}. 

	The first of these two facts can be proved in the usual way: By~\eqref{eq:2137}
	the pairs $\cQ(K_3,X_3)\cS(X_3,L_3)$ 
	and $\cQ(K_3,X_3)\cS(X_3,M_3)$ have $\eps$-intersecting links and, therefore,
	the pairs 
	$\cQ(K_3, X_3)\cT(K_3,L_3)$ and $\cQ(K_3, X_3)\cU(K_3,M_3)$
	have $\eps$-disjoint links. So for arbitrary $k, k'\in K_3$, $x\in X_3$, 
	$\ell\in L_3$, 
	and $m\in M_3$ the subsets $N(Q^{kx}, Q^{k'x})$, $N(T^{k\ell}, T^{k'\ell})$,
	and $N(U^{km}, U^{k'm})$ of $\cP^{kk'}$ have at least the size $(1/3+\eps)|\cP^{kk'}|$
	and the first of them intersects the two other ones in less than $\eps|\cP^{kk'}|$
	vertices each. 
	This yields 
	\[
		|N(T^{k\ell}, T^{k'\ell})\cap N(U^{km}, U^{k'm})|> \eps|\cP^{kk'}|
	\]
	and thus the pair $\cT(K_3, L_3)\cU(K_3,M_3)$ has indeed {$\eps$-intersecting} links.
	
	It is less obvious, however, that the pair $\cU(K_3, M_3)\cV(L_3,M_3)$ has 
	$\eps$-intersecting links as well. To confirm this, we pick arbitrary vertices 
	$k\in K_3$, $x\in X_3$, $\ell\in L_3$, $y\in Y$. Due to $\nu\ll\eps$  
	we can invoke Corollary~\ref{cor:equivalent} 
	and pass to a subset $M_4\subseteq M_3$ of size $2$ such that  
	$\equiv=\equiv_{\eps, M_4}$ is an equivalence relation with at most two 
	equivalence classes on the set of $\nu$-holes
\[
		\bigl\{\Lambda(\cU, M_4, k), \Lambda(\cV, M_4, \ell), 
		\Lambda(\cS, M_4, x), \Lambda(\cS, M_4, y)\bigr\}\,.
	\]

	By the left statement in~\eqref{eq:2137} the pairs 	
	$\cQ(K_3,X_3)\cS(X_3,M_3)$ and $\cQ(L_3,Y_3)\cS(Y_3,M_3)$
	have $\eps$-intersecting links and, hence, by our application of 
	Corollary~\ref{cor:cleaning} to the triples $\cQ\cU\cS$ 
	and $\cQ\cV\cS$ the pairs $\cU(K_3,M_3)\cS(X_3,M_3)$ and 
	$\cV(L_3,M_3)\cS(Y_3,M_3)$ have $\eps$-disjoint links, for which reason 
\begin{equation}\label{eq:0015}
		\Lambda(\cU, M_4, k)\not\equiv \Lambda(\cS, M_4, x)
		\qand
		\Lambda(\cV, M_4, \ell)\not\equiv \Lambda(\cS, M_4, y)\,.
	\end{equation}
Moreover, writing $M_4=\{m, m'\}$ the right part of~\eqref{eq:2137} 
	yields 
	\[
		|N(S^{xm}, S^{xm'})\cap N(S^{ym}, S^{ym'})| 
		\ge 
		\eps\vert \cP^{mm'}\vert\,,
	\]
	whence $\Lambda(\cS, M_4, x)\equiv \Lambda(\cS, M_4, y)$. 
	Together with~\eqref{eq:0015} this discloses 
	\[
		\Lambda(\cU, M_4, k)\equiv \Lambda(\cV, M_4, \ell)
	\]
	and, consequently, $\cU(K_3, M_3)\cV(L_3,M_3)$ has 
	$\eps$-intersecting links, as desired. 
\end{proof}

\subsection{Two large disjoint holes}
\label{sec:bicolourisation}
In this section, we establish the existence of two essentially disjoint holes such that most cherries 
in each hole have a large neighbourhood in the other hole. 
For that we consider the following sets of unwanted cherries. 

Given $\mu$-holes~$\Phi$ and~$\Psi$ on~$J$,~$\gamma>0$, and indices~$ijk\in J^{(3)}$ a cherry~$(P^{ij}, P^{ik})\in \cP^{ij}\times \cP^{ik}$ is~\emph{$\gamma$-bad} 
if either 
\begin{align*}
    (P^{ij}, P^{ik})\in \Phi^{ij}\times \Phi^{ik} \text{ and } \vert N(P^{ij},P^{ik})\setminus \Psi^{jk}\vert&\geq \gamma\vert \cP^{jk}\vert\\
     \text{or}\quad(P^{ij}, P^{ik})\in \Psi^{ij}\times \Psi^{ik} \text{ and }  \vert N(P^{ij},P^{ik})\setminus \Phi^{jk}\vert&\geq \gamma\vert \cP^{jk}\vert\,.
\end{align*}

For~$i<j<k$ we denote the sets of all~$\gamma$-bad \emph{left}, \emph{middle}, and \emph{right} cherries by
\[
	\ccB^{ijk}(\Phi, \Psi, \gamma) \subseteq \cP^{ij}\times\cP^{ik}\,, 
	\quad
	\ccC^{ijk}(\Phi, \Psi, \gamma) \subseteq \cP^{ij}\times\cP^{jk}\,,
	\ \text{and}\
	\ccD^{ijk}(\Phi, \Psi, \gamma) \subseteq \cP^{ik}\times\cP^{jk}\,.
\]

The following lemma shows that given two disjoint holes~$\Phi$ and~$\Psi$ of width at least~$1/3+\eps$ there are either (for a large subset of indices) few $\gamma$-bad cherries 
or there are two other holes covering substantially more space. It might be helpful
to point out that eventually we will only use this lemma for $\gamma=\eps/12$.

\begin{lemma}\label{lem:bad}
For every~$\mu$,~$\eps\geq \gamma>0$ and~$t\in \NN$ there is~$\nu>0$ such that the following holds. Suppose~$\cA$ is an~$\eps$-wicked reduced hypergraph with index set~$I$ and for sufficiently large~$J\subseteq I$ we are given~$\eps$-disjoint 
$(\nu, 1/3+\eps)$-holes~$\Phi$ and~$\Psi$ on $J$. 

Then, there exists a subset~$J_\star\subseteq J$ of size~$t$ such that one of the following alternatives occurs.
\begin{enumerate}[label=\Alabel]
    \item\label{it:newhole}
    There exist two~$\eps$-disjoint~$(\mu,1/3+\eps)$-holes~$\Phi_\star$ and~$\Psi_\star$ 
    on $J_\star$ such that
    \[
    	\vert\Phi_\star^{ij}\cup\Psi_\star^{ij}\vert
    	\geq
    	\vert\Phi^{ij}\cup\Psi^{ij}\vert+\frac{\gamma}{2}\vert\cP^{ij}\vert
	 \]
    for every~$ij\in J_\star^{(2)}$
    \item \label{it:fewbad} or for all~$i<j<k$ in~$J_\star$ the sets of~$\gamma$-bad cherries satisfy
    \begin{multline*}
      \vert \ccB^{ijk}(\Phi, \Psi, \gamma)\vert\leq \mu\vert\cP^{ij}\vert\vert\cP^{ik}\vert\,, 
      \quad 
      \vert \ccC^{ijk}(\Phi, \Psi,\gamma)\vert \leq  \mu\vert\cP^{ij}\vert\vert\cP^{jk}\vert\,, \\
      \text{and}\quad
      \vert \ccD^{ijk}(\Phi, \Psi, \gamma)\vert \leq  \mu\vert\cP^{ik}\vert\vert\cP^{jk}\vert\,.     
    \end{multline*}
\end{enumerate}
\end{lemma}

\begin{proof}
Given~$\mu$,~$\eps\geq \gamma>0$ and $t$ we fix auxiliary 
integers~$t_1, t_2, t_3, t_4$, and we choose~$\nu$ to satisfy
\[
	\eps^{-1}, \mu^{-1}, \gamma^{-1}, t \ll t_4 \ll t_3 \ll t_2 \ll t_1, \nu^{-1}\,.
\]
Let~$\cA$,~$J\subseteq I$,~$\Phi$, and~$\Psi$ be as in the statement of the lemma,
where $J$ is so large that $|J|\lra (t_1,t_1,t_1, t)^3_4$. 
We suppose that~\ref{it:fewbad} fails and intend to derive~\ref{it:newhole}. 

Our assumption on the size of $J$ combined with the failure of~\ref{it:fewbad}   
yields a subset~$J_1\subseteq J$ of size~$t_1$ such that one of the following 
three statements holds:
\begin{enumerate}[label=\nlabel] 
	\item\label{it:1638-1} $\vert \ccB^{ijk}(\Phi, \Psi, \gamma)\vert > 
		\mu\vert\cP^{ij}\vert\vert\cP^{ik}\vert$ for all $i<j<k$ in~$J_1$,
	\item\label{it:1638-2} $\vert \ccC^{ijk}(\Phi, \Psi, \gamma)\vert >
		\mu\vert\cP^{ij}\vert\vert\cP^{jk}\vert$ for all $i<j<k$ in~$J_1$,
	\item\label{it:1638-3} or $\vert \ccD^{ijk}(\Phi, \Psi, \gamma)\vert>
		\mu\vert\cP^{ik}\vert\vert\cP^{jk}\vert$ for all $i<j<k$ in~$J_1$.
\end{enumerate}

As reversing the order $<$ on $I$ exchanges~\ref{it:1638-1} and~\ref{it:1638-3}, 
we may assume that one of the first two cases occurs. 

\smallskip
	\textit{\hskip 0.7em First Case: We have 
	$\vert \ccB^{ijk}(\Phi, \Psi, \gamma)\vert>\mu\,\vert\cP^{ij}\vert\vert\cP^{ik}\vert$ 
		for all $i<j<k$ in~$J_1$.}
\tinyskip

\noindent
For all $i<j<k$ in~$J_1$ at least one of the sets 
\[
	\ccB^{ijk}_\Phi = \ccB^{ijk}(\Phi,\Psi, \gamma)\cap \big(\Phi^{ij}\times\Phi^{ik}\big)		
	\quad\text{and}\quad
	\ccB^{ijk}_\Psi = \ccB^{ijk}(\Phi,\Psi, \gamma)\cap \big(\Psi^{ij}\times\Psi^{ik}\big)
\]
must consist of more than $\frac{\mu}{2}\vert\cP^{ij}\vert\vert\cP^{ik}\vert$ 
bad cherries. 
Thus, a further application of Ramsey's theorem allows us to assume that there is a 
set~$J_2\subseteq J_1$ of size~$t_2$ such that 
\begin{align}\label{eq:ramsey1bad}
    \vert\ccB^{ijk}_\Phi\vert > \frac{\mu}{2} \vert\cP^{ij}\vert\vert\cP^{ik}\vert
\end{align}
holds for all~$i<j<k$ in $J_2$.

\vbox{
\begin{clm}
	There are a set $J_3\subseteq J_2$ of size $t_3$ and transversals $\cQ(J_3)$,
	$\cR(J_3)$ such that 
	\begin{align}\label{eq:omega}
	\vert N(Q^{ij}, R^{ik}) \setminus (\Phi^{jk}\cup\Psi^{jk})\vert 
	\geq 
	\frac{\gamma}{2}\vert \cP^{jk}\vert
\end{align}
for all~$i<j<k$ in~$J_3$ and 
	\begin{align}\label{eq:interneighbours}
    \vert N(Q^{ik}, Q^{jk})\cap N(R^{i\ell}, R^{j\ell})\vert
    \geq 
    \eps\vert \cP^{ij}\vert
	\end{align}
	whenever $i<j<k<\ell$ are in $J_3$.
\end{clm}}

\begin{proof}
Let~$\cA_2$ be the auxiliary reduced hypergraph with index set $J_2$ and
vertex classes~$\cP^{ij}$ for $ij\in J_2^{(2)}$ whose constituents are defined 
by 
\[
	\{P^{ij},P^{ik},P^{jk}\}\in E(\cA_2^{ijk}) 
	\quad\Longleftrightarrow \quad
	(P^{ij},P^{ik})\in \ccB^{ijk}_\Phi\setminus \ccL^{ijk}(\Phi, \gamma/2)
\]
for all $i<j<k$ in $J_2$ and 
all $(P^{ij},P^{ik},P^{jk})\in \cP^{ij}\times \cP^{ik}\times \cP^{jk}$. 
Due to~\eqref{eq:except-cherries} we have 
\[
	\vert \ccL^{ijk}(\Phi, \gamma/2) \vert  
	\leq 
	\frac{2\nu}{\gamma}\vert \cP^{ij}\vert \vert\cP^{ik}\vert 
	\leq 
	\frac\mu4\vert \cP^{ij}\vert \vert\cP^{ik}\vert\
\]
for all~$i<j<k$ in~$J_2$ and together with~\eqref{eq:ramsey1bad} 
this establishes that~$\cA_2$ is~$(\mu/4, \vvv)$-dense. 

Together with
\[
	\vert \ccL^{ijk}(\Phi, \gamma/2) \vert  
	\leq 
	\frac{2\nu}{\gamma}\vert \cP^{ij}\vert \vert\cP^{ik}\vert
	\qqand 
	\vert \ccR^{ijk}(\Phi, \gamma/2) \vert  
	\leq 
	\frac{2\nu}{\gamma}\vert \cP^{ik}\vert \vert\cP^{jk}\vert\,,
\]
and $\nu\ll\gamma, \mu$ this shows that Lemma~\ref{lem:inhabitedtransversals}
yields a set~$J_3\subseteq J_2$ of size~$t_3$ and 
transversals~$\cQ(J_3)$, $\cR(J_3$), and~$\cS(J_3)$ that avoid~$\ccL(\Phi, \gamma/2)$ and~$\ccR(\Phi, \gamma/2)$ and form an inhabited triple in $\cA_2$.

In particular, we have $Q^{ij}R^{ik}S^{jk}\in E(\cA_2)$ 
for all~$i<j<k$ in~$J_3$ and thus 
\begin{align}\label{defQR}
    (Q^{ij},R^{ik})\in\ccB^{ijk}_\Phi\setminus \ccL^{ijk}(\Phi, \gamma/2)\,.    
\end{align}

By the definitions of~$\ccB^{ijk}_\Phi$ and~$\ccL^{ijk}(\Phi, \gamma/2)$ 
this tells us 
\[
	\vert N(Q^{ij}, R^{ik}) \setminus \Psi^{jk}\vert 
	\geq 
	\gamma\vert\cP^{jk}\vert 
	\qqand
	\vert N(Q^{ij}, R^{ik}) \cap \Phi^{jk}\vert 
	< 
	\frac{\gamma}{2}\vert\cP^{jk}\vert\,,
\]
and by subtracting these estimates one easily confirms~\eqref{eq:omega}.

Now let $i<j<k<\ell$ from $J_3$ be arbitrary. 
Since~$\cQ$ and~$\cR$ avoid~$\ccR(\Phi, \gamma/2)$, we have 
\[
	\vert N(Q^{ik}, Q^{jk})\cap \Phi^{ij}\vert
	\leq 
	\frac{\gamma}{2}\vert \cP^{ij}\vert
	< 
	\eps\vert \cP^{ij}\vert 
	\qand 
	\vert N(R^{i\ell}, R^{j\ell})\cap \Phi^{ij}\vert
	\leq 
	\frac{\gamma}{2}\vert \cP^{ij}\vert
	<
	\eps\vert \cP^{ij}\vert\,.
\]

Since each of the three subsets $N(Q^{ik}, Q^{jk})$, $N(R^{i\ell}, R^{j\ell})$, 
and $\Phi^{ij}$ of $\cP^{ij}$ has at least the size $(1/3+\eps)|\cP^{ij}|$, 
this implies~\eqref{eq:interneighbours}.
\end{proof}

Now Lemma~\ref{lem:2148} applied to $J_3$ and the transversals $\cQ(J_3)$, $\cR(J_3)$
yields a set $J_4^+\subseteq J_3$ of size $t_4+1$ such that
all $i<j<k<\ell$ in $J_4^+$ satisfy
\[
   e\bigl(N(Q^{ij}, R^{ik}), N(Q^{ij}, R^{i\ell}),  N(Q^{ik}, R^{i\ell})\bigr)
  	\le 
	\nu |\cP^{jk}||\cP^{j\ell}||\cP^{k\ell}| \,.
\]
Setting $x=\min(J_4^+)$, $J_4=J_4^+\setminus\{x\}$, 
and 
\[
	\Omega^{jk}=N(Q^{xj}, R^{xk})
\] 
for all $jk\in J_4^{(2)}$ we obtain 
\[
	 e(\Omega^{jk}, \Omega^{j\l},  \Omega^{k\l}) 
  	\le 
	\nu|\cP^{jk}||\cP^{j\l}||\cP^{k\l}| 
\]
for all $jk\l\in J_4^{(3)}$. In other words, 
the set $\Omega=\bigdcup_{jk\in J_4^{(2)}}\Omega^{ij}$ 
is a $(\nu, 1/3+\eps)$-hole. Moreover, by~\eqref{eq:omega} we have 
\[
	\big|\Omega^{jk}\setminus(\Phi^{jk}\cup\Psi^{jk})\big|
	\geq
	\frac{\gamma}{2}\big|\cP^{jk}\big|\,.
\]

Now by Corollary~\ref{cor:adhoc} there exists a subset~$J_\star\subseteq J_4$ 
of size~$t$ in which~$\Omega\cup\Phi$ 
and~$\Psi$ or~$\Omega\cup\Psi$ and~$\Phi$ are two~$\eps$-disjoint $\mu$-holes.
Due to~\eqref{eq:omega} this shows that~\ref{it:newhole} holds either 
for~$\Phi_\star = \Phi\cup \Omega $ and~$\Psi_\star=\Psi$, or for~$\Phi_\star=\Phi$ 
and~$\Psi_\star=\Psi\cup\Omega$.

\smallskip
	\textit{\hskip 0.7em Second Case:
	We have  $\vert \ccC^{ijk}(\Phi, \Psi, \gamma)\vert >  \mu\,\vert\cP^{ij}\vert\vert\cP^{jk}\vert$ for all~$i<j<k$ in~$J_1$.}
\tinyskip

\noindent
As before we consider the set of~$\eps$-bad cherries $\ccC_\Phi^{ijk}$ 
and~$\ccC_\Psi^{ijk}$ restricted to the respective 
holes and following the same Ramsey argument 
we find a subset~$J_2\subseteq J$ of size at least~$t_2$ for which we may assume that
\[
	\vert \ccC_\Phi^{ijk}\vert 
	> 
	\frac{\mu}{2}\vert\cP^{ij}\vert \vert \cP^{jk}\vert
\]
holds for all~$i<j<k$ in~$J_2$.

\begin{clm}
	There are a set $J_3\subseteq J_2$ of size $t_3$ and transversals $\cQ(J_3)$,
	$\cS(J_3)$ such that 
	\begin{align}\label{eq:omega-middle}
     \vert N(Q^{ij}, S^{jk}) \setminus (\Phi^{ik}\cup\Psi^{ik})\vert 
     \geq 
     \frac{\gamma}{2}\vert \cP^{ik}\vert
	\end{align}
	whenever $i<j<k$ are in $J_3$, and 	
	\begin{align}\label{eq:interneighbours2}
    \vert N(Q^{ij}, Q^{ik})\cap N(S^{j\ell}, S^{k\ell})\vert
    \geq 
    \eps\vert \cP^{jk}\vert
    \qand
    \vert N(S^{ik}, S^{i\ell})\cap N(S^{jk}, S^{j\ell})\vert
    \geq 
    \eps\vert \cP^{k\ell}\vert
\end{align}
	for all $i<j<k<\ell$ in $J_3$.
\end{clm}

\begin{proof} 
This time the constituents of our auxiliary reduced 
hypergraph~$\cA_2$ with index set~$J_2$ are defined by 
\[
	\{P^{ij},P^{ik},P^{jk}\}\in E(\cA_2^{ijk}) 
	\quad\Longleftrightarrow \quad
	(P^{ij},P^{jk})\in \ccC^{ijk}_\Phi\setminus \ccM^{ijk}(\Phi, \gamma/2)
\]
for~$i<j<k$ in $J_2$ (see Definition \ref{def:exceptional-cherries}). 
As in the first case 
Lemma~\ref{lem:inhabitedtransversals} leads to a set~$J_3\subseteq J_2$ 
of size~$t_3$ and transversals~$\cQ(J_3)$, $\cS(J_3)$ which 
satisfy~$(Q^{ij},S^{jk})\in \ccC^{ijk}_\Phi\setminus \ccM^{ijk}(\Phi, \gamma/2)$ 
for all~$i<j<k$ in~$J_3$ and avoid the left and right $(\gamma/2)$-exceptional 
cherries of $\Phi$. Again the first of these properties yields 
\[
	\vert N(Q^{ij}, S^{jk}) \setminus \Psi^{ik}\vert 
	\geq 
	\gamma\vert\cP^{ik}\vert 
	\qqand
	\vert N(Q^{ij}, S^{jk}) \cap \Phi^{ik}\vert 
	< 
	\frac{\gamma}{2}\vert\cP^{ik}\vert\,,
\]
and~\eqref{eq:omega-middle} follows upon subtraction. 

For the proof~\eqref{eq:interneighbours2} we fix four indices $i<j<k<\ell$ from $J_3$.
The subsets $N(Q^{ij}, Q^{ik})$, $N(S^{j\ell}, S^{k\ell})$, and $\Phi^{jk}$ 
of $\cP^{jk}$ have size at least $(1/3+\eps)|\cP^{jk}|$ and the third of them 
intersects the other two in less than $\eps|\cP^{jk}|$ vertices. This implies 
the left part of~\eqref{eq:interneighbours2}. The right side can be shown in the 
same way, looking at the sets $N(S^{ik}, S^{i\ell})$, $N(S^{jk}, S^{j\ell})$, 
and~$\Phi^{k\ell}$. 
\end{proof}

Now we define 
\[
	\Omega_x^{ij} = N(Q^{ix},S^{xj})\subseteq \cP^{ij}
\]
for all $i<x<y$ in $J_3$. Owing to Lemma~\ref{lem:2150} there exists a set
$J^+_4\subseteq J_3$ of size $2t_4-1$ such that  
\begin{align}\label{badramsey-m-1}
    e(\Omega_x^{ij}, \Omega_x^{ik}, \Omega_y^{jk})
    \le
    \nu \vert\cP^{ij}\vert\vert\cP^{ik}\vert\vert\cP^{jk}\vert
\end{align}
holds for all $i<x<j<y<k$ from $J_4^+$. Let $J_4^+=\{j(1), \dots, j(2t_1-1)\}$
enumerate the elements of $J_4^+$ in increasing order,  
let $J_4=\{j(1), j(3), \dots, j(2t_4-1)\}$ be the $t_4$-element subset of $J_4^+$
consisting of the elements occupying odd positions, and set 
\[
	\Omega^{j(2r-1)j(2s-1)}
	=
	\Omega^{j(2r-1)j(2s-1)}_{j(2r)}
	\quad \text{ for all } rs\in [t_4]^{(2)}\,.
\] 
By~\eqref{badramsey-m-1} the 
set $\Omega=\bigcup_{rs\in [t_4]^{(2)}} \Omega^{j(2r-1)j(2s-1)}$
is a $(\nu, 1/3+\eps)$-hole on $J_4$ and in view of~\eqref{eq:omega-middle}
we can finish as in the first case. 
\end{proof}

Now Lemma~\ref{lem:1247} followed by iterative applications of 
Lemma~\ref{lem:bad} leads to two nonequivalent holes with few bad cherries. 

\begin{cor}\label{cor:nobad}
For every~$\mu$,~$\eps\geq \gamma>0$ and~$t\in \NN$ the following holds. 
If~$\cA$ is an~$\eps$-wicked reduced hypergraph whose index set~$I$ is 
sufficiently large, 
then there exist a subset~$J\subseteq I$ of size~$t$ and~$\eps$-disjoint~$(\mu,1/3+\eps)$-holes~$\Phi$ and~$\Psi$ on $J$ such that
for all~$i<j<k$ in~$J$ the sets of~$\gamma$-bad cherries satisfy
    \begin{multline}
        \vert \ccB^{ijk}(\Phi, \Psi, \gamma)\vert 
        \leq
        \mu\vert\cP^{ij}\vert\vert\cP^{ik}\vert\,, 
        \quad 
        \vert \ccC^{ijk}(\Phi, \Psi, \gamma)\vert 
        \leq 
        \mu\vert\cP^{ij}\vert\vert\cP^{jk}\vert\,,\\
        \qand\quad 
        \vert \ccD^{ijk}(\Phi, \Psi, \gamma)\vert 
        \leq 
        \mu\vert\cP^{ik}\vert\vert\cP^{jk}\vert
        \,. \end{multline} 
\end{cor}

\begin{proof}
By Lemma~\ref{lem:bad} there are functions $f\colon \RR_{>0}\times \NN\lra \RR_{>0}$
	and $g\colon \RR_{>0}\times \NN\lra \NN$ such that for all $t_\star\in \NN$,
	$\mu_\star\in \RR_{>0}$ the conclusion of Lemma~\ref{lem:bad} holds for 
	$\mu_\star$, $t_\star$, $f(\mu_\star, t_\star)$ and $g(\mu_\star, t_\star)$
	here in place of $\mu$, $t$, $\nu$ and $|J|$ there. 
	
	Starting with $\mu_0\!=\!\mu$ and $t_0=t$ we recursively set $\mu_{m+1}=f(\mu_m, t_m)$
	and \mbox{$t_{m+1}=g(\mu_m, t_m)$} for every integer $m\geq 0$.
	Without loss of generality we may assume that the sequence $(\mu_m)_{m\ge 0}$
	is decreasing and that $t_m\ge 2$ for every~$m$. 
	Setting $s=\lceil 4\gamma^{-1}\rceil$ we shall now prove the conclusion of 
	our corollary for $|J|\gg t_s, \mu_s^{-1}, \eps^{-1}$. 
	
	Due to Lemma~\ref{lem:1247} there are a set 
	$J_s\subseteq J$ of size $t_s$ and two $\eps$-disjoint $(\mu_s, 1/3+\eps)$-holes 
	on $J_s$. Thus there exists a least nonnegative integer $m\le s$ such that 
	there are a set $J_m\subseteq J$ of size $t_m$ and two $\eps$-disjoint 
	$(\mu_m, 1/3+\eps)$-holes $\Phi$, $\Psi$ on $J_m$ such that 
	$|\Phi^{ij}|+|\Psi^{ij}|> (s-m)\gamma |\cP^{ij}|/2$ holds for every 
	pair $ij\in J_k^{(2)}$. 
	
	As our choice of $s$ entails $s\gamma/2\ge 2$, we cannot have $m=0$. 
	Thus Lemma~\ref{lem:bad} leads to a set $J_{m-1}\subseteq J_m$ of size $t_{m-1}$
	such that either~\ref{it:newhole} or~\ref{it:fewbad} holds for $\mu_{m-1}$
	here in place of $\mu$ there. By the minimality 
	of $m$ alternative~\ref{it:newhole} is impossible. For this reason the restrictions
	of $\Phi$ and $\Psi$ to arbitrary $t$-element subsets of $J_{m-1}$ 
	are as desired.
\end{proof}

\subsection{Bicolourisation}
\label{ssec:pf-reduction}
It remains to argue that by taking a random preimage we can convert 
Corollary~\ref{cor:nobad} into Proposition~\ref{prop:reduction}.

\begin{proof}[Proof of Proposition~\ref{prop:reduction}]
Given~$\eps$ and~$t$ we take~$\gamma, \mu>0$ and~$\ell\in\NN$ such that
\[
	\gamma = \frac{\eps}{12} \qqand \eps, t^{-1} \gg \ell^{-1} \gg \mu
\]
and consider an $\eps$-wicked reduced hypergraph $\cA$ whose index set $I$ 
is sufficiently large. 
Due to Corollary~\ref{cor:nobad} there are a set~$J\subseteq I$ of size~$t$ 
and~$\eps$-disjoint~$\mu$-holes~$\Phi$, $\Psi$ on $J$ such that 
for all~$i<j<k$ in~$J$ we have 
\begin{multline}\label{fewbad}
    \vert \ccB^{ijk}(\Phi, \Psi, \gamma)\vert 
    \leq
    \mu\vert\cP^{ij}\vert\vert\cP^{ik}\vert\,, 
    \quad 
    \vert \ccC^{ijk}(\Phi, \Psi, \gamma)\vert 
    \leq
    \mu\vert\cP^{ij}\vert\vert\cP^{jk}\vert\,, \\
    \text{and}\quad 
    \vert \ccD^{ijk}(\Phi, \Psi, \gamma)\vert
    \leq
    \mu\vert\cP^{ik}\vert\vert\cP^{jk}\vert\,.
\end{multline}

Next we define a reduced subhypergraph~$\cA_1$ of $\cA$ admitting a bicolouring $\phi_1$
which satisfies, with only few exceptions, the minimum codegree 
condition $\tau_2(\cA_1, \phi_1)\ge 1/3+\eps/4$. To this end we consider for every 
pair~$ij\in J^{(2)}$ the sets  
\[
	\fR^{ij} = \Phi^{ij}\setminus \Psi^{ij} 
	\qand 
	\fB^{ij} = \Psi^{ij}\setminus \Phi^{ij}\,,
\]
and then we set~$\fR=\bigdcup_{ij\in J^{(2)}} \fR^{ij}$ 
as well as~$\fB=\bigdcup_{ij\in J^{(2)}} \fB^{ij}$.
Now let $\cA_1$ be the reduced hypergraph with index set~$J$, vertex 
classes~$\cP_1^{ij}=\fR^{ij}\dcup \fB^{ij}\subseteq \cP^{ij}$ for every~$ij\in J^{(2)}$, and edges
\[
	E(\cA_1)=E(\fR\cup\fB)\setminus  \bigl(E(\fR)\cup E(\fB)\bigr)\,.
\]
Since $\Phi$ and $\Psi$ are $\eps$-disjoint and have width at least $1/3+\eps$ we have 
\begin{equation}\label{eq:sP1}
	\big|\cP_1^{ij}\big|
	\geq 
	\Big(\frac{2}{3}+\eps\Big)\big|\cP^{ij}\big|
\end{equation}
for every $ij\in J^{(2)}$.

It is plain that the map $\phi_1\colon V(\cA_1)\lra\{\tred, \tblue\}$
defined by $\phi_1^{-1}(\tred)=\fR$ and $\phi_1^{-1}(\tblue)=\fB$ is 
a bicolouring of $\cA_1$.

\begin{clm}\label{clm:1517}
	In $\cA_1$ all monochromatic cherries $(P_1^{ij}, P_1^{ik})$ 
	that fail to be~$\gamma$-bad in $\cA$ have codegree  at least $(1/3+\eps/4)|\cP_1^{jk}|$.
\end{clm} 

\begin{proof}
Suppose~$i<j<k$ 
and that~$(P_1^{ij}, P_1^{ik})=(R^{ij},R^{ik})\in \fR^{ij}\times \fR^{ik}$ 
is a red left cherry not belonging to $\ccB^{ijk}(\Phi,\Psi,\gamma)$. 
Due to 
\[
	\vert N_\cA(R^{ij},R^{ik})\setminus\Psi^{jk}\vert
	\leq
	\gamma\vert\cP^{jk}\vert
\]
we have 
\begin{align*}
    \vert N_{\cA_1}(R^{ij},R^{ik})\vert 
    &=\!
    \vert N_{\cA}(R^{ij},R^{ik})\cap \fB^{jk} \vert \\
    &\geq\!
    \vert N_{\cA}(R^{ij},R^{ik})\vert -
    \vert N_{\cA}(R^{ij},R^{ik})\setminus \Psi^{jk} \vert -
    \vert\Phi^{jk}\cap\Psi^{jk}\vert \\
    &\geq\!
    \left(\frac{1}{3}+\eps\right)\vert\cP^{jk}\vert -\gamma\vert\cP^{jk}\vert -
    \vert\Phi^{jk}\cap\Psi^{jk}\vert  \\
    &\ge\!
    \left(\frac13+\frac\eps4\right)\left(|\cP^{jk}|-|\Phi^{jk}\cap\Psi^{jk}|\right)
    +\frac23\left(\eps|\cP^{jk}|-|\Phi^{jk}\cap\Psi^{jk}|\right) \\
    &\ge\! 
    \left(\frac13+\frac\eps4\right)|\cP_1^{jk}|\,,
\end{align*}
where the penultimate inequality uses the definition of~$\gamma$ and 
the last inequality follows 
from~$\cP_1^{jk} \subseteq \cP^{jk}\setminus(\Phi^{jk}\cap\Psi^{jk})$. 
This concludes the proof for red left cherries and all other cases can 
be treated analogously.
\end{proof}

Similarly as in~\cite{Christiansurvey}*{Lemma~4.2} we will define the 
reduced hypergraph~$\cA_\star$ by taking the preimage of a random homomorphism  
$h\in\fA(\cA_1,\l)$. Recall from Definition~\ref{def:randompre}
that for every map $h\in\fA(\cA_1,\l)$ the 
associated reduced hypergraph~$\cA_h$ has index set $J$ and vertex classes 
$\cPb^{ij}$ of size $\ell$.

Observe that there is no $h\in\fA(\cA_1,\l)$ such that $\cA_h$ supports a~$K_5^{(3)}$,
because otherwise the homomorphism $h$ would show that~$\cA_1\subseteq \cA$ 
supports a~$K_5^{(3)}$ as well, contrary to $\cA$ being wicked.  
Furthermore, for every $h\in\fA(\cA_1,\l)$ the map $\phi_h=\phi_1\circ h$ is 
a bicolouring of~$\cA_h$. So it remains to show that if $h$ gets chosen uniformly 
at random, then with positive probability the event
\[
    \tau_2(\cA_h,\phi_h)\geq \frac{1}{3}+\frac{\eps}{8}
\]
occurs. We estimate for each cherry of $\cA_h$ the probability that it violates 
this condition.

\begin{clm} \label{clm:1526}
	If $ijk\in J^{(3)}$ and $(\Pb^{ij},\Pb^{ik})\in \cPb^{ij}\times\cPb^{ik}$ 
	is a cherry of $\cA_h$, then
	the event~$\cX$ that 
	\[
	\phi_h(\Pb^{ij}) = \phi_h(\Pb^{ik}) 
	\qand 
	\vert N_{\cA_h}(\Pb^{ij},\Pb^{ik})\vert
		< 
	\Big(\frac{1}{3}+\frac{\eps}{8}\Big)\vert\cPb^{jk}\vert
	\]
	has at most the probability $3\mu + \exp\big({-}\tfrac{\eps^2\ell}{128}\big)$.
\end{clm}

\begin{proof}
	Without loss of generality we may assume~$i<j<k$. By the law of total probability we have
		\[
    \mathds P(\cX) 
    =\frac{1}{\vert \cP_1^{ij} \vert \vert \cP_1^{ik}\vert}\sum_{(P^{ij},P^{ik})\in\cP_1^{ij}\times\cP_1^{ik}}
    \mathds P\left(\cX\,\big\vert\,h(\Pb^{ij})=P^{ij} \tand h(\Pb^{ik})=P^{ik}\right)\,.
   \]
   	Note that cherries $(P^{ij},P^{ik})$ consisting of two vertices with different 
	colours contribute zero to this sum. Moreover, due to \eqref{fewbad} the total
	contribution from cherries $(P^{ij},P^{ik})$ belonging 
	to $\ccB^{ijk}(\Phi,\Psi, \gamma)$
	is at most 
	\[
		\frac{\mu |\cP^{ij}||\cP^{ik}|}{|\cP_1^{ij}||\cP_1^{ik}|}
		\overset{\eqref{eq:sP1}}{\le} 
		\left(\frac 32\right)^2\mu
		<
		3\mu\,. 
	\]
	Furthermore, for $P^{ij}$, $P^{ik}$ of the 
	same colour with $(P^{ij},P^{ik})\not\in \ccB^{ijk}(\Phi,\Psi, \gamma)$
	Claim~\ref{clm:1517} combined with Chernoff's inequality tells us
		\[
		\mathds P\left(\cX\,\big\vert\,h(\Pb^{ij})=P^{ij} \tand h(\Pb^{ik})=P^{ik}\right)
		\le 
		\exp\big({-}\tfrac{\eps^2\ell}{128}\big)
	\]
	and Claim~\ref{clm:1526} follows.
\end{proof}

Since $\cA_h$ has~$3\ell^2\binom{t}{3}$ cherries, Claim~\ref{clm:1526} implies
\begin{align*}
    \mathds P\left(\tau_2(\cA_h, \phi_h) < \frac{1}{3}+\frac{\eps}{8}\right) 
    \leq 
    3\l^2\binom{t}{3}\left(3\mu + \exp\big({-}\tfrac{\eps^2\ell}{128}\big)\right)\,.
\end{align*}
Owing to the hierarchy $\mu\ll \ell^{-1} \ll t^{-1}$ this probability is smaller 
than~$1$ and, therefore, there is a map~$h\in\fA(\cA_1,\l)$ for which~$\cA_h$ has the desired properties. 
\end{proof}

\section{Cliques on five vertices in bicoloured reduced hypergraphs}
\label{sec:bicoloured}
In this section we establish Proposition \ref{prop:bicolored} and show that bicoloured reduced hypergraphs with minimum monochromatic codegree density bigger than $1/3$ support a $K_5^{(3)}$.

In the proof we shall use the following types of neighbourhoods in reduced hypergraphs~$\cA$. 
For two vertices $P$, $P'\in V(\cA)$ and a subset $\cU\subseteq V(\cA)$ we denote by $N_{\cU}(P,P')$ the neighbourhood restricted to~$\cU$. Similarly, for two subsets 
$\cU$, $\cU'\subseteq V(\cA)$ we write $N_{\cU\times\cU'}(P)$ for the set of 
pairs in $\cU\times\cU'$ that together with $P$ form a hyperedge in $\cA$, i.e.,
\begin{align*}
	N_{\cU}(P,P') &=\{U\in \cU\colon PP'U\in E(\cA)\}
	\\
	\qand N_{\cU\times\cU'}(P)&=\{(U,U')\in \cU\times\cU'\colon PUU'\in E(\cA)\}\,.
\end{align*}

\begin{proof}[Proof of Proposition \ref{prop:bicolored}]
Clearly we may assume that $\eps<\frac 16$. 
Fix a sufficiently small auxiliary constant $\xi$ with $0<\xi\ll\eps$ 
such that~$\frac{1/6-\eps}{\xi}$ is equal to some positive integer~$s$.
Moreover, let $I$ be a sufficiently large index set such that its cardinality 
satisfies the partition relation $|I|\lra(5)^2_s$, meaning that it is at least 
as large as the $s$-colour Ramsey number for the graph clique $K_5$.  
Let~$\mathcal A$ be a bicoloured reduced hypergraph with index set~$I$ and vertex 
classes $\cP^{ij}$ for $ij\in I^{(2)}$ and let the bicolouring 
$\phi\colon V(\cA)\to\{\tred,\tblue\}$ satisfy~$\tau_2(\mathcal A, \phi)\geq 1/3 + \eps$.

For every $ij\in I^{(2)}$ we set
\[
	\Red^{ij}=\phi^{-1}(\tred)\cap \mathcal P^{ij} 
	\qqand \rho_{ij}=\frac{\vert \Red^{ij}\vert}{\vert \mathcal P^{ij}\vert}
\]
and, analogously, we define $\Blue^{ij}=\phi^{-1}(\tblue)\cap\mathcal P^{ij}$ and 
$\beta_{ij}=\vert \Blue^{ij}\vert/\vert \mathcal P^{ij}\vert$.
In view of~\eqref{eq:nontrivial}, the assumption on $\tau_2(\mathcal A, \phi)$ 
implies that all $\rho_{ij}$, $\beta_{ij}$ are in $[1/3+\eps,2/3-\eps]$.
Splitting this interval into~$s$ intervals of length~$2\xi$, the size of~$I$ yields a subset $J\subseteq I$ of size $5$ such that all~$\beta_{ij}$ with $ij\in J^{(2)}$ are in the same interval.
Let~$\beta$ be the centre of this interval and set~$\rho=1-\beta$. We thus arrive at
\[
    \beta_{ij} = \beta \pm \xi \qqand \rho_{ij} = \rho \pm \xi
\]
for all $ij\in J^{(2)}$.
Without loss of generality we may assume~$\beta \leq \rho$, which implies
\begin{align}\label{eq:rhobeta}
    \frac13 +\eps\leq \beta-\xi < \beta \leq \frac12\leq \rho< \rho+\xi \leq \frac23-\eps\,.
\end{align}

For $ijk\in J^{(3)}$ the codegree condition translates for $\tred$ vertices $R^{ij}\in\fR^{ij}$
and $R^{ik}\in\fR^{ik}$ to
\begin{align}
	|N_{\fB^{jk}}(R^{ij},R^{ik})|
	=
	d(R^{ij},R^{jk})
	&\geq
	\left(\frac{1}{3}+\eps\right)\vert \mathcal P^{jk}\vert\nonumber\\
	&\geq
	\left(\frac{1}{3}+\eps\right)\left(\frac{1}{\beta+\xi}\right)\vert \fB^{jk}\vert
	\geq
	\left(\frac{1}{3\beta}+\frac{\eps}{2}\right)\vert \fB^{jk}\vert\,,\label{eq:codegreeR}
\end{align}
where we used $\xi\ll\eps,\beta$ for the last inequality. Similarly, for $\tblue$ vertices we have
\begin{equation}
	|N_{\fR^{jk}}(B^{ij},B^{ik})| 
	\geq
	\left(\frac{1}{3\rho}+\frac{\eps}{2}\right)\vert \fR^{jk}\vert\,.\label{eq:codegreeB}
\end{equation}

We may rename the indices in $J$ and assume that $J=\ZZ/5\ZZ$.
We shall show that~$\cA$ restricted to $J$ supports a $K_5^{(3)}$. For that we have to find
ten vertices~$P^{ij}\in \mathcal P^{ij}$, one for every~$ij\in J^{(2)}$, such that for all of the ten triples 
$ijk\in J^{(3)}$ the vertices $P^{ij}$, $P^{ik}$, and~$P^{jk}$ span a hyperedge in the constituent~$\mathcal A^{ijk}$. For every $i\in J=\ZZ/5\ZZ$ we will select~$P^{i,i+1}$ from~$\fB^{i,i+1}$ 
and~$P^{i,i+2}$ from~$\fR^{i,i+2}$. 
Since $\cA$ contains no monochromatic triples as hyperedges, this choice for the colour classes is up to a permutation of indices unavoidable, as it corresponds to the unique $2$-colouring of $E(K_5)$ with no monochromatic triangle.

The rest of the proof is based on several averaging arguments relying on the minimum degree condition. 
For generic vertices from $\fR$ and $\fB$ we shall use capital letters~$R$ and $B$. In the process,  
we will make appropriate choices to fix the ten special vertices that induce the
supported $K_5^{(3)}$. For those vertices, we will use small letters $r$ and $b$ depending on their colour.  

We begin with the selection of $r^{14}\in\fR^{14}$.
Applying~\eqref{eq:codegreeB} to all pairs of vertices $B^{15}\in\Blue^{15}$ and $B^{45}\in\Blue^{45}$
implies that the total number of hyperedges in $\cA^{145}$ 
crossing the sets~$\Red^{14}$, $\Blue^{15}$,  and~$\Blue^{45}$ is at least
\[
	\vert\Blue^{15}\vert\vert \Blue^{45}\vert 
		\cdot \left(\frac{1}{3\rho}+\frac{\eps}{2}\right)\vert \Red^{14}\vert\,.
\] 
Consequently, we can fix some vertex~$r^{14}\in \Red^{14}$ such that
\begin{align}\label{eq:G1}
                        |N_{\Blue^{15}\times \Blue^{45}}(r^{14})|
    \geq 
    \left(\frac{1}{3\rho}+\frac{\eps}{2}\right)\vert\Blue^{15}\vert\vert\Blue^{45}\vert\,.
\end{align}
The following claim fixes the four vertices $b^{12}$, $b^{34}$ and $r^{13}$, $r^{24}$. 

\begin{clm}\label{fact}
There are blue vertices~$b^{12}\in\Blue^{12}$, $b^{34}\in\Blue^{34}$ 
and red vertices $r^{13}\in\Red^{13}$,~${r^{24}\in\Red^{24}}$ such that
\begin{enumerate}[label=\rmlabel]
\item\label{it:factedge1} $b^{12}r^{14}r^{24}$ and $r^{13}r^{14}b^{34}$ are hyperedges in $\cA$
\item\label{it:fact}and $\vert N_{\Blue^{23}}(b^{12}, r^{13})\cap N_{\Blue^{23}}(r^{24},b^{34}) \vert \geq \bigl(1-\frac{1}{3\beta}\bigr)\vert \Blue^{23}\vert$.
\end{enumerate}
\end{clm}

\begin{proof}
Owing to~\eqref{eq:codegreeR} for every~$R^{13}\in \Red^{13}$ we have 
$d(R^{13},r^{14})\geq \big(\frac{1}{3\beta}+\frac{\eps}{2}\big)\vert \Blue^{34}\vert$
and, hence, there is a vertex~$b^{34}\in \Blue^{34}$ such that
\begin{align}\label{eq:13}
    |N_{\Red^{13}}(r^{14}, b^{34})|
    \geq 
    \left(\frac{1}{3\beta}+\frac{\eps}{2}\right)\vert \Red^{13}\vert 
    \geq
    \frac{\rho}{3\beta}\vert \mathcal P^{13}\vert\,.
        \end{align}

Similarly, we can fix a vertex~$r^{24}\in\Red^{24}$ such that
\begin{align}\label{eq:23}
    \vert N_{\Blue^{23}}(r^{24}, b^{34}) \vert
    \geq 
    \frac{1}{3\rho}\vert \Blue^{23}\vert\,. 
\end{align}
Recalling that~$\vert\Red^{13}\vert \leq (\rho+\xi )\vert \cP^{13}\vert$ for every $B^{12}\in \fB^{12}$ and  $B^{23}\in \fB^{23}$ we have \begin{align*}
    \big\vert N_{\Red^{13}}(B^{12},B^{23})\cap N_{\Red^{13}}(r^{14}, b^{34}) \big\vert 
    &\overset{\phantom{\eqref{eq:13}}}{\geq}
    \left(\frac{1}{3}+\eps\right)\vert\mathcal P^{13}\vert
    	+\big\vert N_{\Red^{13}}(r^{14}, b^{34})\big\vert 
		-\vert \Red^{13}\vert  \\
	&\overset{\phantom{\eqref{eq:13}}}{\geq}
    \big\vert N_{\Red^{13}}(r^{14}, b^{34})\big\vert 
    	- \left(\rho+\xi-\frac{1}{3}-\eps\right)\vert\mathcal P^{13}\vert\\
	&\overset{\eqref{eq:13}}{\geq}
	\left(1-3\beta+\frac{\beta}{\rho}\right)\big\vert N_{\Red^{13}}(r^{14}, b^{34})\big\vert \\
    &\overset{\phantom{\eqref{eq:13}}}{\geq}
    \left(3\rho-\frac\rho\beta\right) \big\vert N_{\Red^{13}}(r^{14}, b^{34})\big\vert\,, 
\end{align*}
where the last estimate uses $\beta+\rho=1$ and $\frac\beta\rho+\frac\rho\beta\ge 2$.
Hence, the number of hyperedges crossing $N_{\Blue^{12}}(r^{14},r^{24})$, 
$N_{\Blue^{23}}(r^{24}, b^{34})$, and~$N_{\Red^{13}}(r^{14},b^{34})$ is at least 
\[
	\vert N_{\Blue^{12}}(r^{14},r^{24})\vert\vert N_{\Blue^{23}}(r^{24}, b^{34}) \vert
	\cdot\left(3\rho-\frac\rho\beta\right)
	\vert N_{\Red^{13}}(r^{14},b^{34})\vert\,.
\]
Consequently, there exist $b^{12}\in N_{\Blue^{12}}(r^{14}, r^{24})$
and $r^{13}\in N_{\Red^{13}}(r^{14}, b^{34})$ such that 
\begin{align*}
    \vert N_{\Blue^{23}}(b^{12},r^{13})\cap N_{\Blue^{23}}(r^{24},b^{34}) \vert 
    &\overset{\phantom{\eqref{eq:23}}}{\geq}
    \left(3\rho-\frac\rho\beta\right)
    \vert N_{\Blue^{23}}(r^{24}, b^{34})\vert\\
    &\overset{\eqref{eq:23}}{\geq}
    \left(1-\frac 1{3\beta}\right)\vert \Blue^{23}\vert\,. \qedhere
\end{align*}
\end{proof}

The next claim fixes the four vertices $b^{15}$, $b^{45}$ and $r^{25}$, $r^{35}$. Together with Claim~\ref{fact}
this fixes all vertices except $b^{23}$ and both claims together guarantee those seven hyperedges supporting a~$K_5^{(3)}$
that do not involve $b^{23}$.
\begin{clm}\label{fact2}
There exist blue vertices~$b^{15}\in\Blue^{15}$, $b^{45}\in\Blue^{45}$ and red vertices $r^{25}\in\Red^{25}$,
$r^{35}\in\Red^{35}$ such that $b^{12}b^{15}r^{25}$, $r^{13}b^{15}r^{35}$, $r^{14}b^{15}b^{45}$, $r^{24}r^{25}b^{45}$, and $b^{34}r^{35}b^{45}$ are hyperedges in~$\cA$.
\end{clm}
\begin{proof}
We consider two sets of ``candidates'' for the pair $(b^{15}, b^{45})$ that are 
relevant for the existence of $r^{25}$ and $r^{35}$. More precisely, we set
\begin{align*}
    G_1&
    =
    \{(B^{15},B^{45})\in \Blue^{15}\times \Blue^{45}
    \colon
    N_{\Red^{25}}(b^{12},B^{15})\cap N_{\Red^{25}}(r^{24}, B^{45}) \neq \emptyset\}\\ 
    \qand
    G_2&
    =
    \{(B^{15},B^{45})\in \Blue^{15}\times \Blue^{45}
    \colon
    N_{\Red^{35}}(r^{13},B^{15})\cap N_{\Red^{35}}(b^{34}, B^{45}) \neq \emptyset\}\,.
\end{align*}
Note that for every~$B^{15}\in \Blue^{15}$ there is some  
$R^{25} \in N_{\Red^{25}}(b^{12}, B^{15})$ and we have
\[
	\vert N_{\Blue^{45}}(r^{24},R^{25})\vert 
	\overset{\eqref{eq:codegreeR}}{\geq} 
	\frac{1}{3\beta}\vert\Blue^{45}\vert\,.
\]
Clearly,~$\{B^{15}\}\times N_{\Blue^{45}}(r^{24},R^{25})\subseteq G_1$ and, hence, we establish 
\begin{align}\label{eq:G23}
    \vert G_1\vert    
    \geq
    \frac{1}{3\beta}\vert\Blue^{15}\vert\vert\Blue^{45}\vert\,.
\end{align}
A symmetric argument yields the same bound for~$G_2$. Combining~\eqref{eq:G23} 
and the same bound for $G_2$ with~\eqref{eq:G1} leads to
\begin{align*}
    \vert G_1\vert +\vert G_2\vert+|N_{\Blue^{15}\times \Blue^{45}}(r^{14})|
    \geq
    \left(\frac{2}{3\beta}+\frac{1}{3\rho}+\frac{\eps}{2}\right)
    \vert\Blue^{15}\vert\vert\Blue^{45}\vert 
    \overset{\eqref{eq:rhobeta}}{>}
    2\,\vert\Blue^{15}\vert\vert\Blue^{45}\vert\,.
\end{align*}
Consequently, we can fix a pair~$(b^{15}, b^{45})\in G_1\cap G_2 \cap N_{\Blue^{15}\times \Blue^{45}}(r^{14})$.
Moreover, having fixed~$b^{15}$ and $b^{45}$ this yields a vertex~$r^{25}\in\fR^{25}$ from the non-empty intersection considered in the definition of $G_1$. Similarly, $G_2$ leads to our choice of $r^{35}\in\fR^{35}$.

Since $(b^{15}, b^{45})\in N_{\Blue^{15}\times \Blue^{45}}(r^{14})$, 
the hyperedge $r^{14}b^{15}b^{45}$ exists in $\cA$ and the other four 
hyperedges result from the definitions of $G_1$ and $G_2$.
\end{proof}

As mentioned above, Claims~\ref{fact} and~\ref{fact2} fix  all vertices
except~$b^{23}\in \Blue^{23}$ and all hyperedges not involving $b^{23}$.
For the three remaining hyperedges it suffices to show that 
\[
	N_{\Blue^{23}}(b^{12}, r^{13})\cap N_{\Blue^{23}}(r^{24},b^{34})\cap N_{\Blue^{23}}(r^{25}, r^{35})
	\neq \emptyset\,.
\]
Claim~\ref{fact}\,\ref{it:fact} and~\eqref{eq:codegreeR} imply 
\begin{multline*}
    \big\vert N_{\Blue^{23}}(b^{12}, r^{13})\cap N_{\Blue^{23}}(r^{24},b^{34})\cap N_{\Blue^{23}}(r^{25}, r^{35})\big\vert \\
    \geq 
    \big\vert N_{\Blue^{23}}(b^{12}, r^{13})\cap N_{\Blue^{23}}(r^{24},b^{34})\big\vert 
    	+ \big\vert N_{\Blue^{23}}(r^{25}, r^{35})\big\vert 
		- \big\vert \fB^{23}\big\vert \\
    \overset{\eqref{eq:codegreeR}}{\geq}
    \left(1-\frac{1}{3\beta}+\frac{1}{3\beta}+\frac{\eps}{2}-1\right) \vert \Blue^{23}\vert 
    > 
    0\,.
\end{multline*}
Hence a suitable choice for $b^{23}$ exists and, therefore,~$\cA$ restricted to $J$ supports a $K_5^{(3)}$.
\end{proof}

\section{Concluding Remarks}
We close with a few related open problems and possible directions for future research.

\subsection{Tur\'an problems for cliques in 
	\texorpdfstring{$\robustee$}{cherry}-dense hypergraphs}	
In view of Theorems~\ref{thm:K2r} and~\ref{thm:main} for cliques $K_\l^{(3)}$ with 
$\l\leq 16$ vertices only the cases $\l=9$ and $10$ are still unresolved and closing the bounds
\[
	\frac{1}{2}
	\leq 
	\pi_{\ee}(K_9^{(3)})
	\leq 
	\pi_{\ee}(K_{10}^{(3)})
	\leq 
	\frac{2}{3}
\]
would be interesting. We believe that by combining our main result with the 
ideas in~\cite{cherry} one can derive 
an improved upper bound $\pi_{\ee}(K_{10}^{(3)})\le \frac35$. More generally, 
it seems  that $\pi_{\ee}(K_{r}^{(3)})=\alpha$ 
implies $\pi_{\ee}(K_{2r}^{(3)})\le\frac 1{2-\alpha}$ and we shall return to this 
topic in the near future. 

Determining the value $\pi_{\ee}(K_\l^{(3)})$ for large values of $\l$ 
might be a challenging problem and one may first focus on the asymptotic behaviour.
For every $\l\geq 3$ Theorem~\ref{thm:K2r} tells us 
\begin{equation}\label{eq:asymp-ub}
	\pi_{\ee}(K_\l^{(3)})
	\leq
	1-\frac{1}{\log_2(\l)}\,.
\end{equation}
For a lower bound, we consider the following well-known random construction.

\begin{exmp}\label{ex:Ramsey}\sl 
For $r\geq 2$ we consider random hypergraphs $H_{\phi}=(V,E_{\phi})$ 
with the edge set defined by the non-monochromatic triangles of a random $r$-colouring 
$\phi\colon V^{(2)}\to [r]$ for a sufficiently large vertex set $V$. 
It is easy to check that for any fixed $\eta>0$ with high probability 
such hypergraphs $H_{\phi}$ are $(\eta,\frac{r-1}{r},\ee)$-dense. On the other hand, 
if $\l$ is at least as large as $R(3;r)$, the $r$-colour Ramsey number for graph triangles, 
then every such $H_{\phi}$ is~$K_{\l}^{(3)}$-free. 
\end{exmp}

Consequently, Example~\ref{ex:Ramsey} yields
\[
	\pi_{\ee}(K_\l^{(3)})
	\geq
	1-\frac{1}{r}\,,\ \text{whenever}\ \l\geq R(3;r)
\]
and using the simple upper bound $R(3;r)\leq 3\,r!$ 
we arrive at 
\begin{equation}\label{eq:asymp-lb}
	\pi_{\ee}(K_\l^{(3)})
	\geq
	1-\frac{\log_2\log_2(\l)}{\log_2(\l)}
\end{equation}
for sufficiently large $\l$. Comparing the bounds in~\eqref{eq:asymp-ub} and~\eqref{eq:asymp-lb}
leads to the following problem.
\begin{prob}\label{prob:asymp}
	Determine the  asymptotic behaviour of $1-\pi_{\ee}(K_\l^{(3)})$.
\end{prob}
Obtaining sharp bounds for $R(3;r)$ is a well known open problem in Ramsey theory. 
The aforementioned upper bound on $R(3;r)$ together with the work of Schur~\cite{Sch16}
yields 
\[
	2^{\Omega(r)}
	\leq
	R(3;r)
	\leq
	2^{O(r\log r)}\,.
\]
The intriguing problem of deciding whether $\displaystyle{\lim_{r\to\infty}}R(3;r)^{1/r}$ is finite or not is attributed to Erd\H os (see, e.g.,~\cite{ChGr98}*{Section~2.5}). In case this limit would turn out to be finite, then this would yield 
a sharper estimate for~\eqref{eq:asymp-lb}. In fact, this would close the gap to~\eqref{eq:asymp-ub} and 
the resolution of Problem~\ref{prob:asymp} would be $1-\pi_{\ee}(K_\l^{(3)})=\Theta(1/\log(\l))$.

\subsection{Tur\'an problems for hypergraphs with uniformly dense links}
As discussed in the introduction there is a small difference between Theorem~\ref{thm:main}
and Corollary~\ref{cor:main}. Below we briefly elaborate on these differences.

In this work we study $\ee$-dense hypergraphs, which are defined by the lower bound 
condition~\eqref{eq:eedef} in Definition~\ref{def:ee}. Requiring in addition a matching upper bound, 
i.e., replacing~\eqref{eq:eedef} by 
\[
	\big|e_{\ee}(P,Q)-d\,|\cK_{\ee}(P,Q)|\big|
	\leq
	\eta |V|^3\,,
\]
leads to the notion of \emph{$(\eta,d,\ee)$-quasirandom} hypergraphs. Clearly, 
we can transfer the definition of $\pi_{\ee}(F)$ in \eqref{eq:pieedef}
and define the Tur\'an-density $\pi'_{\ee}(F)$ by restricting to $\ee$-quasirandom 
hypergraphs~$H$
\begin{multline*}
	\pi'_{\ee}(F) 
	= 
	\sup\{d\in[0,1] \colon \text{for every $\eta>0$ and $n\in\NN$ there exists an $F$-free,}\\
		\text{$(\eta, d, \ee)$-quasirandom hypergraph with at least $n$ vertices}\}\,.
\end{multline*}
By definition we have $\pi'_{\ee}(F)\leq \pi_{\ee}(F)$
for every hypergraph $F$ and one may wonder if this inequality is sometimes strict.

For $K_5^{(3)}$ it is easy to check that the lower bound construction in Example~\ref{ex:K5}
yields $K_5^{(3)}$-free $(\eta,1/3,\ee)$-quasirandom hypergraphs for every fixed $\eta>0$ and, hence,
\[
	\pi'_{\ee}(K^{(3)}_5)
	=
	\pi_{\ee}(K^{(3)}_5)=\frac{1}{3}\,.
\]
On the other hand, the lower bound construction for $K^{(3)}_6$ from~\cite{cherry}
is given by Example~\ref{ex:Ramsey} for $r=2$.
In those hypergraphs $H_{\phi}$ we can take $P$ and $Q$ to be the pairs in colour $1$ and $2$ respectively and get 
\[
	e_{\ee}(P,Q)=|\cK_{\ee}(P,Q)|\,,
\]
i.e., they have relative density~$1$.
Therefore, the hypergraphs $H_\phi$ are only $(\eta,1/2,\ee)$-dense, but not $(\eta,1/2,\ee)$-quasirandom.
In fact, we are not aware of any matching quasirandom lower bound construction for 
$\piee(K^{(3)}_6)$ and it seems possible that $\pi'_{\ee}(K^{(3)}_6)$ is strictly smaller 
than $\pi_{\ee}(K^{(3)}_6)$ suggesting the following general problem.\footnote{We remark that for the concepts of $\vvv$-dense/quasirandom hypergraphs 
there is no difference for the corresponding Tur\'an-densities, as every $\vvv$-dense 
hypergraph contains large $\vvv$-quasirandom hypergraphs of at least the same density.}
\begin{prob}
	Which hypergraphs $F$ satisfy $\pi'_{\ee}(F)<\pi_{\ee}(F)$?
\end{prob}

\subsection*{Acknowledgement} We are indebted to the referees for their careful work and 
efforts towards improving our presentation.

\end{document}